\documentclass[11pt,a4paper]{article}
\usepackage[margin=2.5cm]{geometry}
\usepackage{amsmath,amssymb,amsthm}
\usepackage{enumitem}
\usepackage{latexsym}
\usepackage{hyperref}
\hypersetup{colorlinks=true,linkcolor=blue,citecolor=blue}
 
\newtheorem{theorem}{Theorem}[section]
\newtheorem{lemma}[theorem]{Lemma}
\newtheorem{corollary}[theorem]{Corollary}
\newtheorem{proposition}[theorem]{Proposition}
\newtheorem{remark}[theorem]{Remark}

\title{The Local Four-Square Problem over \(\mathbb{Z}_{p^k}\)}
\author{H. Orelma\footnote{Tampere Institute of Mathematics, Hepolamminkatu 51, 33720 Tampere, Finland\\
E-Mail: heikki.orelma@proton.me, heikki.orelma@mail.ru}}
\date{\today}

\begin{document}

\maketitle

\begin{abstract}
The norm map \(N:\mathcal{H}_{\mathbb{Z}_{p^k}}\to \mathbb{Z}_{p^k}\) is studied on the quaternion ring over \(\mathbb{Z}_{p^k}\), where \(p\) is an odd prime and $k\ge 1$ an integer. By means of the isomorphism \(\mathcal{H}_{\mathbb{Z}_{p^k}}\cong M_2(\mathbb{Z}_{p^k})\), quaternions are investigated using matrix methods. It is shown that the fibre size
\[
a_{p^k}(m)=|\{q\in \mathcal{H}_{\mathbb{Z}_{p^k}}:N(q)=m\}|
\]
depends only on the \(p\)-adic valuation \(v_p(m)\) of \(m\). Explicit formulas for the fibre sizes are derived for every \(m\in\mathbb{Z}_{p^k}\):
\[
a_{p^k}(m)=
\begin{cases}
p^{3k-2}(p^2-1), & t=0,\\[6pt]
p^{3k-2-t}(p+1)(p^{t+1}-1), & 0<t<k,\\[6pt]
p^{2k-1}(p^{k+1}+p^k-1), & t=k,
\end{cases}
\]
where \(t=v_p(m)\) (with the convention \(v_p(0)=k\)).

The main result of the paper is a complete solution to the \emph{local four-square problem} over the ring \(\mathbb{Z}_{p^k}\): the number \(a_{p^k}(m)\) gives the exact number of quadruples \((x_1,x_2,x_3,x_4)\in\mathbb{Z}_{p^k}^4\) satisfying
\[
m=x_1^2+x_2^2+x_3^2+x_4^2.
\]
The proof is purely algebraic; it relies only on matrix theory and Smith normal form, thus avoiding the abstract machinery of number theory.
\end{abstract}

\section{Introduction}
In this work, quaternions are considered, which are generally represented as an extension of the complex numbers in the form
\begin{align}\label{kvat}
q=w+ai+bj+ck,    
\end{align}
where $i,j,k$ are symbols resembling imaginary units and $w,a,b,c$ are scalars.

The first encounter with the number system now known as the quaternions was made by Olinde Rodrigues (1795--1851), who studied the displacement of a rigid body between two positions in his work published in 1840 \cite{Rodrigues1840}. While examining the displacement of a body between two successive rotations, he stated on page 409:
\begin{quotation}
    "Dans le cas du changement de l'ordre des rotations, le signe de $\cos L$ devient négatif..."
\end{quotation}
Although his work was purely geometric, an algebra of rotations was described in which the noncommutativity of rotations and the symmetry of the rotation axis played central roles. In modern terminology, the quaternions (\ref{kvat}) with $w=1$ were implicitly identified by Rodrigues, although the theory was not developed further from an algebraic point of view.

A systematic and algebraic approach was introduced by William Hamilton (1805--1865) in his work published in 1844 \cite{Hamilton1844}. The quaternions were presented in the form (\ref{kvat}), and the famous multiplication rules for the imaginary units were established:
\[
i^2=j^2=k^2=ijk=-1.
\]
The coefficients were taken to satisfy $w,a,b,c\in\mathbb{R}$. The name ``quaternions'' was introduced for these quadruples, and a completely new algebraic system was constructed in which the geometric observations of Rodrigues were given a precise algebraic formulation and broad applications.

From the perspective of the theory developed in this article, the next step was taken in the paper published by Rudolf Lipschitz (1832--1903) in 1886 \cite{Lipschitz1886}, where quaternions and their integer arithmetic were studied for the first time, that is, quaternions of the form (\ref{kvat}) with coefficients $w,a,b,c\in\mathbb{Z}$. However, the theory of Lipschitz did not lead to a complete arithmetic theory, since unique factorization does not hold in this system.

The theory of Lipschitz was completed by Adolf Hurwitz (1859--1919) in his work published in 1896 \cite{Hurwitz1896} (see also \cite{Hurwitz1919}). It was observed that if either $w,a,b,c\in \mathbb{Z}$ or $w,a,b,c\in \frac{1}{2}+\mathbb{Z}$, then a number system is obtained for which the Euclidean algorithm and unique factorization are valid. This number system is now known as the \textit{Hurwitz quaternions}.

The work of Lipschitz and Hurwitz was brought together by Leonard Dickson (1874--1954), and the theory was systematized and further developed; see \cite{Dickson1923b,Dickson1923a,Dickson1901, Dickson1924}. It was shown, among other things, that the factorization theorems for the Lipschitz and Hurwitz quaternions can be proved without the use of a division algorithm, thereby demonstrating that the Euclidean algorithm is not a necessary condition for unique factorization. The same method was also applied by Dickson to more general hypercomplex systems in which no division algorithm is available.

Modern ring theory became established during the 1920s and 1930s as part of the axiomatization of algebra, particularly through the work of Richard Dedekind, Emmy Noether, and Emil Artin. This development made it natural to consider quaternions (\ref{kvat}) whose coefficients $w,a,b,c$ belong to an arbitrary ring $R$. In this work, this line of research in modern algebra is continued by investigating the structure of quaternions and the norm map over general rings.\\

In number theory, particularly in the study of the integers, one of the classical problems of interest has been the so-called \textit{four square problem}, which asks whether, for a given positive integer $m\in\mathbb{Z}_+$, the equation
\begin{align}\label{NeljNelProb}
x_1^2 + x_2^2 + x_3^2 + x_4^2 = m
\end{align}
admits integer solutions $x_1,x_2,x_3,x_4\in\mathbb{Z}$. In 1770, it was proved by Joseph-Louis Lagrange (1736--1813) that, for every $m$, a solution quadruple $(x_1,x_2,x_3,x_4)$ exists (see, for example, \cite[Section 20.5]{HardyWright}). Later, in 1829, an explicit formula was given by Carl Gustav Jacob Jacobi (1804--1851) for the exact number of ordered integer solutions (including negative integers and zero) of equation (\ref{NeljNelProb}):
\[
r_4(m) = 8 \sum_{\substack{d \mid m \\ 4 \nmid d}}d,
\]
where the summation is taken over all positive divisors $d$ of $m$ that are not divisible by four (see again \cite[Section 20.5]{HardyWright}).

Quaternions were employed by Dickson in \cite{Dickson1924}, where a new descent proof of Lagrange's four-square theorem was presented that is simpler than the earlier proofs.

In this article, a local version of the four square problem (\ref{NeljNelProb}) is investigated. By means of quaternions and their matrix representation, it is proved that, for every $m\in\mathbb{Z}_{p^k}$, where $p$ is an odd prime and $k\ge 1$ is an integer, equation (\ref{NeljNelProb}) admits solutions $x_1,x_2,x_3,x_4\in\mathbb{Z}_{p^k}$, and explicit formulas are derived for the numbers of solutions $a_{p^k}(m)$. The method developed is based on the algebraic structure of the quaternion ring and the investigation of the norm map, and it provides a natural framework for the study of analogous local norm representation problems in other algebraic structures.

\section{Quaternions over Commutative Rings}
Classically, the quaternion algebra $\mathbb{H}$ is defined by imposing the multiplication rules
\begin{align}\label{laskus}
i^2=j^2=k^2=ijk=-1,
\end{align}
on the generators $\{1,i,j,k\}$, after which a general quaternion is represented in the form
\[
q=w+xi+yj+zk,
\]
where $w,x,y,z\in\mathbb{R}$ are assumed. Quaternions may be generalized as a $4$-dimensional algebra in essentially two ways. Either the multiplication rules (\ref{laskus}) may be generalized (see, for example, \cite{Yaglom}), or number systems other than the real numbers may be used for the coefficients $w,x,y,z$. In this paper, the latter approach is adopted.\\
\\
In this section, a comprehensive introduction to quaternions over rings is provided. Although the principal new results of this article are obtained for the ring $R=\mathbb{Z}_{p^k}$, the material in this section is presented for general commutative rings $R$ so that readers encountering this area of mathematics for the first time may obtain a broader understanding of the subject. Most of the material presented in this section is well known, and further information may be found in the references.

\subsection{The Quaternion Ring $\mathcal{H}_R$}
It is assumed that $R$ is a unital associative commutative ring, and
\[
\mathcal{H}_R=\{ q=w+xi+yj+zk\ :\ w,x,y,z\in R\},
\]
is defined. This set forms a ring over $R$ when the generators $\{1,i,j,k\}$ are assumed to satisfy the multiplication rules (\ref{laskus}). The ring $\mathcal{H}_R$ is referred to as the \textit{quaternion ring over $R$}. Natural examples include the classical quaternion algebra $\mathcal{H}_{\mathbb{R}}=\mathbb{H}$ and the complexified quaternion algebra $\mathcal{H}_{\mathbb{C}}=\mathbb{H}_{\mathbb{C}}$.\\
\\
It is well known that, for example, when $R=\mathbb{R}$, the quaternion ring $\mathcal{H}_R$ is a division ring. In general, however, the situation is somewhat more complicated. The (multiplicative) \textit{group of units} of the ring $R$ is defined by
\[
U(R)=\{ u\in R\ :\ \exists u^{-1}\in R,\ uu^{-1}=u^{-1}u=1\},
\]
which consists of all invertible elements of $R$. For a general element
\[
q=w+xi+yj+zk
\]
of $\mathcal{H}_R$, the \textit{conjugate} is defined by
\[
\overline{q}=w-xi-yj-zk.
\]
As in the classical quaternion algebra $\mathbb{H}$, the identities
\[
\overline{\overline{q}}=q\qquad\text{and}\qquad \overline{pq}=\overline{q}\; \overline{p}
\]
may be proved for all $p,q\in \mathcal{H}_R$. The norm map $N:\mathcal{H}_R\to R$ is defined by
\[
N(q)=q\overline{q}=\overline{q}q=w^2+x^2+y^2+z^2.
\]
The group of units of the quaternion ring $\mathcal{H}_R$ is then given by
\[
U(\mathcal{H}_R)=\{q\in \mathcal{H}_R\ :\ N(q)\in U(R)\},
\]
that is, if $q\in U(\mathcal{H}_R)$, then
\[
q^{-1}=N(q)^{-1}\overline{q}.
\]
A necessary and sufficient condition for $\mathcal{H}_R$ to be a division ring is given by the following result.

\begin{proposition}
    $\mathcal{H}_R$ is a division ring if and only if the following conditions are satisfied:
    \begin{enumerate}
        \item $R$ is a field,
        \item $N:\mathcal{H}_R\to R$ is anisotropic.
    \end{enumerate}
\end{proposition}
\begin{proof}
Since $R\hookrightarrow \mathcal{H}_R$ is naturally embedded as the real part, a necessary condition for $\mathcal{H}_R$ to be a division ring is that every $u\in R$ be invertible, that is, $R$ must be a field. If $q\in\mathcal{H}_R$ is a nonzero element satisfying $N(q)=q\overline{q}=0$, then $q$ is a zero divisor. A sufficient condition for $\mathcal{H}_R$ to be a division ring is that $N(q)=0$ imply $q=0$, that is, that the norm be anisotropic.
\end{proof}
By the preceding proposition, it follows that the quaternion ring
$\mathcal{H}_R$ is a division ring whenever $R=\mathbb{R}$ or any of its subfields
(for example, $R=\mathbb{Q}$), since the norm map $N$
is anisotropic in this case.\\
\\
On the other hand, over the field of complex numbers $\mathbb{C}$ the norm is not anisotropic, since, for example,
\[
i^2+1^2+0^2+0^2=0.
\]
Therefore, $\mathcal{H}_{\mathbb{C}}$ is not a division ring.\\
\\
If $\mathbb{F}_q$ is the finite field with $q$ elements of characteristic $p$ (that is, $q=p^r$), then, by Chevalley's theorem (see \cite{Chevalley}), the equation
\[
w^2+x^2+y^2+z^2=0
\]
always admits nontrivial solutions, since the polynomial has degree $2$ and there are $4$ variables. Consequently, there exists a nonzero element $q \in \mathcal{H}_{\mathbb{F}_q}$ satisfying $N(q)=0$, and hence $q$ is a zero divisor. Therefore, $\mathcal{H}_{\mathbb{F}_q}$ is not a division ring.

\subsection{The Jacobson Radical}
In particular, for commutative rings $R$ that are not fields, one of the fundamental tools for studying the structure of the ring is the so-called Jacobson radical.\\

An \textit{ideal} is a nonempty subset $I\subseteq R$ that is closed under subtraction and multiplication by elements of the ring:
\[
a,b\in I \implies a-b\in I,\qquad a\in I,\ r\in R \implies ra\in I.
\]

A \textit{maximal ideal} is a proper ideal $M\subsetneq R$ for which no proper ideal $I$ exists such that
\[
M\subsetneq I\subsetneq R.
\]
A classical characterization states that $M$ is a maximal ideal if and only if the quotient ring $R/M$ is a field; see, for example, \cite[Proposition 12, p.~254]{DummitFoote2004}. The set of maximal ideals of $R$ is denoted by $\operatorname{Max}(R)$.\\

The \textit{Jacobson radical} is the ideal defined as the intersection of all maximal ideals of the ring:
\[
J(R) = \bigcap_{M \in \operatorname{Max}(R)} M.
\]
If $R$ is a field, then its only ideals are $\{0\}$ and $R$. Since $R/\{0\}\cong R$, the ideal $\{0\}$ is maximal, and hence the Jacobson radical is trivial, namely $J(R)=\{0\}$. Therefore, the Jacobson radical is particularly useful when $R$ is not a field, since in the field case it reduces to the zero ideal.\\

A ring $R$ is said to be \textit{local} if it possesses a unique maximal ideal $\operatorname{Max}(R)=\{M\}$, that is, $J(R)=M$. In particular, for a local ring the quotient ring $R/J(R)$ is always a field. Local rings are especially useful in the theory of quaternion rings.

\begin{proposition}\label{Nollajakat}
    Let $R$ be a local ring. Then
    \[
    U(R)= R\setminus J(R).
    \]
\end{proposition}
\begin{proof}
If $R$ is a field, then $U(R)=R\setminus\{0\}$ and $J(R)=\{0\}$, so the result holds. Assume that $J(R)\neq \{0\}$.
Let $u\in U(R)$ and suppose that $u$ belongs to the radical $J(R)$. By assumption, there exists $u^{-1}\in R$ such that $uu^{-1}=1\in J(R)$, since $J(R)$ is an ideal. It follows that $r1=r\in J(R)$ for every $r\in R$, and hence $J(R)=R$. This is a contradiction, since the Jacobson radical is a proper ideal. Therefore,
\[
U(R)\cap  J(R)=\emptyset.
\]
Now assume that $a\notin J(R)$. Consider the ideal $\langle a\rangle =\{ ar : r\in R\}$. If this ideal were proper, then it would be contained in some maximal ideal, and hence $\langle a\rangle\subseteq J(R)$, a contradiction. Therefore, $\langle a\rangle$ is not a proper ideal, so $\langle a\rangle=R$. Consequently, $1\in \langle a\rangle$, that is, $ab=1$ for some $b\in R$, and hence $a\in U(R)$.
\end{proof}

\section{Matrix Representation}
Let the ring of \(2\times2\) matrices over the ring \(R\) be denoted by
\[
M_2(R)=
\left\{
\begin{pmatrix}
a&b\\
c&d
\end{pmatrix}
:\;
a,b,c,d\in R
\right\}.
\]
The question of when $\mathcal{H}_R$ can be represented by means of \(2\times2\) matrices over $R$ is considered next, that is, when $\mathcal{H}_R$ is isomorphic to $M_2(R)$. Naturally, the result depends on the structure of the ring $R$. For the real quaternions, the result does not hold:
\[
\mathcal H_{\mathbb R}\not\cong M_2(\mathbb R),
\]
since $\mathcal H_{\mathbb R}$ is a division ring. For the complexified quaternions, the following isomorphism holds:
\[
\mathcal H_{\mathbb C} \cong M_2(\mathbb C).
\]
The following theorem provides a sufficient condition for the existence of such an isomorphism.

\begin{theorem}\label{Matest}
Let \(R\) be a commutative ring such that \(2\in U(R)\). If the equation
\[
1+\alpha^2+\beta^2=0,
\]
admits a solution in the ring $R$, then
\[
\mathcal H_R\cong M_2(R).
\]
\end{theorem}

\begin{proof}
The matrices
\[
I=\begin{pmatrix}1&0\\0&1\end{pmatrix},\qquad
A=\begin{pmatrix}\alpha&\beta\\ \beta&-\alpha\end{pmatrix},\qquad
B=\begin{pmatrix}0&1\\-1&0\end{pmatrix}
\]
are defined.
From the condition $1+\alpha^2+\beta^2=0$, it follows that
\[
A^2=-I.
\]
Moreover,
\[
B^2=-I,\qquad AB=-BA.
\]
Since the matrices $A$ and $B$ satisfy the defining relations of the quaternions, the mapping
\[
1\mapsto I,\qquad i\mapsto A,\qquad j\mapsto B
\]
extends uniquely to an $R$-algebra homomorphism
\[
\varphi:\mathcal H_R\to M_2(R).
\]
Consequently,
\[
\varphi(k)=\varphi(ij)=AB.
\]
It remains to be shown that $\varphi$ is bijective. Since both $\mathcal H_R$ and $M_2(R)$ are free $R$-modules of rank $4$, it suffices to show that the images
\[
\varphi(1)=I,\quad \varphi(i)=A,\quad \varphi(j)=B,\quad \varphi(k)=AB
\]
are linearly independent over $R$.
Let $a,b,c,d\in R$ satisfy
\[
aI+bA+cB+dAB=0.
\]
By substituting the matrices and comparing the matrix entries, the system
\[
\begin{aligned}
a+b\alpha-d\beta&=0,\\
a-b\alpha+d\beta&=0,\\
b\beta+c+d\alpha&=0,\\
b\beta-c+d\alpha&=0
\end{aligned}
\]
is obtained.
By adding and subtracting these equations, it follows that $2a=0$ and $2c=0$. Since $2\in U(R)$, it follows that $a=c=0$. The remaining equations are
\[
b\alpha=d\beta,\qquad b\beta=-d\alpha.
\]
These may be written in matrix form as
\[
\begin{pmatrix}
\alpha & -\beta\\
\beta & \alpha
\end{pmatrix}
\begin{pmatrix}
b\\ d
\end{pmatrix}
=
\begin{pmatrix}
0\\0
\end{pmatrix}.
\]
Since
\[
\det
\begin{pmatrix}
\alpha&-\beta\\
\beta&\alpha
\end{pmatrix}
=
\alpha^2+\beta^2
=
-1,
\]
and $-1$ is always a unit, the matrix is invertible. Hence $b=d=0$. Therefore $a=b=c=d=0$, and the images are linearly independent. Consequently, the homomorphism is injective, since its kernel is trivial.

Since $M_2(R)$ is a free $R$-module of rank $4$, every set of four linearly independent elements forms a basis. Therefore, the images $\varphi(1),\varphi(i),\varphi(j),\varphi(k)$ span the entire module $M_2(R)$, and hence $\varphi$ is surjective.
\end{proof}
The preceding theorem establishes the isomorphism not only in the complex case $R=\mathbb{C}$ but also, for example, for finite fields $\mathbb{F}_q$, where, by the well-known theorem of Mann (see \cite{Mann}), the equation
\[
-1=\alpha^2+\beta^2
\]
always admits a solution.\\
\\
The result in \cite[Theorem 3.10]{Ghosseiri19} is a special case of Theorem \ref{Matest}. There it is assumed that $R$ is an algebra over a field $F$ of odd characteristic. The theorem above shows that, for the existence of the isomorphism, it is sufficient that the equation
\[
1+\alpha^2+\beta^2=0
\]
admits a solution in the ring $R$, and that no additional assumptions concerning the algebraic structure are required.

\begin{proposition}\label{Normi}
Assume that $R$ is a ring satisfying the assumptions of Theorem \ref{Matest}. Then
\[
N(q)=\det(\varphi(q))
\]
for every \(q\in\mathcal H_R\).
\end{proposition}
\begin{proof}
Let
\[
\varphi:\mathcal H_R\longrightarrow M_2(R)
\]
be the isomorphism constructed in Theorem \ref{Matest}. Let \(q = w + xi + yj + zk \in \mathcal H_R\). Then
\[
\varphi(q) =
\begin{pmatrix}
w + x\alpha - z\beta & x\beta + y + z\alpha \\
x\beta - y + z\alpha & w - x\alpha + z\beta
\end{pmatrix}.
\]
The determinant is computed as
\begin{align*}
    \det(\varphi(q))
    &=(w + x\alpha - z\beta)(w - x\alpha + z\beta)-(x\beta - y + z\alpha)(x\beta + y + z\alpha)\\
    &=w^2+y^2-\big((x\alpha - z\beta)^2+(x\beta   + z\alpha)^2\big)\\
    &=w^2+y^2-(\alpha^2+\beta^2)(x^2+z^2).
\end{align*}
Since $\alpha^2+\beta^2=-1$, it follows that
\[
\det(\varphi(q))=w^2+x^2+y^2+z^2=N(q).
\]
\end{proof}

\begin{proposition}\label{RadikProp}
Assume that $R$ is a ring satisfying the assumptions of Theorem \ref{Matest}. Then
   \[
J(\mathcal H_R) \cong M_2(J(R)).
\]
\end{proposition}
\begin{proof}
Let $\varphi:\mathcal{H}_R\to M_2(R)$ be the isomorphism given in Theorem \ref{Matest}. Let $M$ be a maximal left (or right) ideal of the ring $\mathcal{H}_R$. It follows directly from the definition of an ideal that $M':=\varphi(M)$ is a maximal left (or right) ideal of the ring $M_2(R)$. Hence,
 \[
\varphi\big(J(\mathcal{H}_R)\big)=\varphi\Big(\bigcap_{M \in \operatorname{Max}(\mathcal{H}_R)} M\Big)=\bigcap_{M' \in \operatorname{Max}(M_2(R))} M'=J(M_2(R)).
 \]
Therefore,
\[
J(\mathcal H_R) \cong J(M_2(R)).
\]
The Jacobson radical of a matrix ring satisfies
\[
J(M_2(R)) = M_2(J(R)),
\]
see \cite[Example (7), pp.~57--58]{Lam}.
\end{proof}

\subsection{Structure of the Unit Group}

The group
\[
\mathrm{GL}_2(R)=\{ A\in M_2(R) : \exists B\in M_2(R) \text{ such that } AB=BA=I_2 \},
\]
is defined, where $R$ is a ring.

\begin{proposition}\label{EksJon}
Let \(R\) be a commutative ring satisfying \(\mathcal H_R \cong M_2(R)\). Then there exists a short exact sequence
\[
1 \longrightarrow 1 + J(\mathcal H_R) \longrightarrow U(\mathcal H_R) \longrightarrow \mathrm{GL}_2(R/J(R)) \longrightarrow 1.
\]
\end{proposition}
The following lemma is proved first.

\begin{lemma}
Let $S$ be an arbitrary ring. Then there exists a short exact sequence
\[
1 \longrightarrow 1+J(S) \longrightarrow U(S) \longrightarrow U(S/J(S)) \longrightarrow 1,
\]
where the maps are the natural inclusion $1+J(S)\hookrightarrow U(S)$ and the quotient homomorphism $U(S)\to U(S/J(S))$.
\end{lemma}

\begin{proof}
Let $\psi: U(S) \to U(S/J(S))$ be the map defined by $\psi(u)=u+J(S)$, so that $\psi(u)\psi(v)=\psi(uv)$ for all $u,v\in U(S)$. This map is well defined, since the image of the group of units is contained in the group of units of the quotient ring. It remains to be shown that the image is precisely the group of units of the quotient ring.\\

Let $\widetilde{u} \in U(S/J(S))$. Then there exists $\widetilde{v}$ such that
\[
\widetilde{u} \widetilde{v} = \widetilde{1}.
\]
Hence,
\[
uv - 1 \in J(S),
\]
that is,
\[
uv = 1 + x
\]
for some $x\in J(S)$. By \cite[Corollary 4.5]{Lam}, it is known that $uv=1+x\in U(S)$. Since $uv$ is a unit, $u$ possesses a right inverse. Similarly, a left inverse is obtained. Therefore, $\psi$ is surjective.\\

Now suppose that $u\in \ker{\psi}$, that is, $\psi(u)=1+J(S)$, or equivalently, $u-1\in J(S)$. Hence,
\[
u=1+(u-1)\in 1+J(S),
\]
and therefore
\[
\ker{\psi} = \{ u\in U(S) : u-1 \in J(S) \} = 1+J(S),
\]
which, as the kernel of a group homomorphism, is a subgroup of $U(S)$.\\

The preceding arguments establish the short exact sequence
\[
1 \longrightarrow 1+J(S) \longrightarrow U(S) \longrightarrow U(S/J(S)) \longrightarrow 1,
\]
where the first map is the inclusion and the second is the quotient homomorphism.
\end{proof}

\begin{proof}[Proof of Proposition \ref{EksJon}]
Let $R$ be a commutative ring satisfying $\mathcal H_R \cong M_2(R)$. In the preceding lemma, let $S=\mathcal H_R$. Then the short exact sequence
\[
1 \longrightarrow 1+J(\mathcal H_R) \longrightarrow U(\mathcal H_R) \longrightarrow U(\mathcal H_R/J(\mathcal H_R)) \longrightarrow 1
\]
is obtained. Let $\varphi:\mathcal H_R \to M_2(R)$ be a ring isomorphism. By Proposition \ref{RadikProp} and its proof, the Jacobson radical is preserved under the isomorphism, that is,
\[
\varphi(J(\mathcal H_R)) = J(M_2(R)) = M_2(J(R)).
\]
Consequently, for every $u\in \mathcal{H}_R$,
\[
\varphi(u+J(\mathcal H_R))=\varphi(u)+M_2(J(R)),
\]
and hence
\[
\mathcal{H}_R/J(\mathcal H_R) \cong M_2(R)/M_2(J(R)).
\]
By the classical isomorphism theorem for matrix rings (see, for example, \cite[Example (6), p.~244]{DummitFoote2004}), it follows that
\[
M_2(R)/M_2(J(R)) \cong M_2(R/J(R)).
\]
Therefore,
\[
U(\mathcal H_R/J(\mathcal H_R)) \cong U(M_2(R/J(R))) = \mathrm{GL}_2(R/J(R)).
\]
\end{proof}

\section{The residue class ring modulo $p^k$}
The remainder of this article is devoted to the ring
\[
R = \mathbb{Z}_{p^k} = \mathbb{Z}/p^k\mathbb{Z},
\]
where \(p\) is an odd prime and \(k \ge 1\). For $a,b\in \mathbb{Z}_{p^k}$, the addition and multiplication are induced by the
congruence relations
\[
ab=c\quad\text{and}\quad a+b=d,
\]
where $ab \equiv c \pmod{p^k}$ and $a+b \equiv d \pmod{p^k}$. In particular, $p^t=0$ in the ring $\mathbb{Z}_{p^k}$ whenever $t=k,k+1,\ldots$, since
\begin{align}\label{potemzo}
    p^t \equiv 0 \pmod{p^k}
\end{align}
because $p^k\mid p^t$.

\begin{proposition}\label{YksRom}
Let \(p\) be a prime and \(k \ge 1\). Then
\begin{enumerate}
    \item[(a)] $U(\mathbb{Z}_{p^k}) = \{ a \in \mathbb{Z}_{p^k} : \gcd(a, p) = 1 \}$,
    \item[(b)] $|U(\mathbb{Z}_{p^k})| = p^{k-1}(p-1)$,
    \item[(c)] $U(\mathcal{H}_{\mathbb{Z}_{p^k}}) = \{ q \in \mathcal{H}_{\mathbb{Z}_{p^k}} : N(q) \not\equiv 0 \pmod{p} \}$.
\end{enumerate}
\end{proposition}
\begin{proof}
Part (a). Let \(a \in \mathbb{Z}_{p^k}\).\\
\\
Assume that \(a \in U(\mathbb{Z}_{p^k})\). Then there exists \(b \in \mathbb{Z}_{p^k}\) such that
\[
ab \equiv 1 \pmod{p^k}.
\]
Since $p$ is prime, either \(\gcd(a, p)= 1\) or \(\gcd(a, p)=p\).
If the latter holds, then \(p \mid a\), and hence \(p \mid ab\), so that \(ab \equiv 0 \pmod{p}\). On the other hand, the congruence \(ab \equiv 1 \pmod{p^k}\) implies, in particular, that \(ab \equiv 1 \pmod{p}\), which is a contradiction.\\

Assume that \(\gcd(a, p) = 1\). Then \(\gcd(a, p^k) = 1\), since the only prime divisor of \(p^k\) is \(p\). By Bézout's identity, there exist integers \(u, v \in \mathbb{Z}\) such that
\[
au + p^k v = 1.
\]
It follows that \(au \equiv 1 \pmod{p^k}\), so that \(u\) is the inverse of \(a\) modulo \(p^k\). Therefore, \(a \in U(\mathbb{Z}_{p^k})\).\\
\\
Part (b). Since the ring \(\mathbb{Z}_{p^k}\) contains exactly \(p^k\) elements, those divisible by \(p\) are of the form
\[
0, p, 2p, 3p, \dots, (p^{k-1}-1)p,
\]
that is, there are exactly \(p^{k-1}\) such elements. The remaining elements are units. Hence,
\[
|U(\mathbb{Z}_{p^k})| = p^k - p^{k-1} = p^{k-1}(p-1).
\]
Part (c). It is known that, in the quaternion ring \(\mathcal{H}_R\),
\[
q \in U(\mathcal{H}_R) \iff N(q) \in U(R).
\]
By part (a), it follows that
\[
N(q) \in U(\mathbb{Z}_{p^k}) \iff \gcd(N(q), p) = 1 \iff p \nmid N(q) \iff N(q) \not\equiv 0 \pmod{p}.
\]
\end{proof}
The ideals of the ring \(\mathbb Z_{p^k}\) are
\[
\langle p^i\rangle,\qquad i=0,1,\ldots,k,
\]
which form the descending chain
\[
\mathbb Z_{p^k}=\langle 1\rangle\supset \langle p\rangle \supset \langle p^2\rangle \supset\cdots \supset \langle p^k\rangle=\langle 0\rangle.
\]
Each ideal is of the form
\[
\langle p^i\rangle = \{ 0, p^i, 2p^i, 3p^i, \ldots, (p^{k-i}-1)p^i \}.
\]
Hence,
\begin{align}\label{ideaalinkoko}
    |\langle p^i\rangle| = p^{k-i}.
\end{align}
The unique maximal ideal of the ring $\mathbb{Z}_{p^k}$ is
\[
\langle p\rangle=\{0,p,2p,\ldots,p^k-p\},
\]
and therefore $\mathbb Z_{p^k}$ is a local ring whose Jacobson radical is
\[
J(\mathbb Z_{p^k})=\langle p\rangle.
\]
Thus, $\mathbb Z_{p^k}$ is a commutative principal ideal ring, and since, by Proposition \ref{Nollajakat},
\begin{align}\label{NULJOK}
U(\mathbb Z_{p^k})=\mathbb Z_{p^k}\setminus \langle p\rangle,
\end{align}
it follows that all zero divisors (that is, precisely the elements divisible by $p$) belong to the Jacobson radical.

\begin{proposition}\label{Aarkund}
Let \(p\) be a prime and \(k \ge 1\). Then
\[
\mathbb{Z}_{p^k}/J(\mathbb{Z}_{p^k})=\mathbb{Z}_{p^k}/\langle p\rangle\cong\mathbb{Z}_p,
\]
where $\mathbb{Z}_p$ is a finite field.
\end{proposition}

\begin{proof}
The mapping
\[
\varphi:\mathbb{Z}_{p^k}\longrightarrow\mathbb{Z}_p,
\qquad
x \bmod p^k \longmapsto x \bmod p
\]
is defined.
This map is a well-defined and surjective ring homomorphism.
Moreover,
\[
\ker\varphi
=
\{x\in\mathbb{Z}_{p^k}:x\equiv0\pmod p\}
=\langle p\rangle.
\]
Since \(J(\mathbb{Z}_{p^k})=\langle p\rangle\), it follows that
\[
\ker\varphi=J(\mathbb{Z}_{p^k}).
\]
By the First Isomorphism Theorem,
\[
\mathbb{Z}_{p^k}/J(\mathbb{Z}_{p^k})
\cong
\mathbb{Z}_{p^k}/\langle p\rangle
\cong
\mathbb{Z}_p.
\]
\end{proof}

\begin{corollary}\label{nojakkaaa}
The following statements are equivalent:
\begin{itemize}
    \item[(a)] $a\in U(\mathbb{Z}_{p^k})$,
    \item[(b)] $a\notin J(\mathbb{Z}_{p^k})=\langle p\rangle$,
    \item[(c)] $a\not\equiv 0 \mod p$.
\end{itemize}
\end{corollary}
The matrix representation of the ring $\mathcal H_{\mathbb{Z}_{p^k}}$ is considered next.

\begin{proposition}\label{Pkmatriisi}
Let $p$ be an odd prime and $k\ge 1$. Then
\[
\mathcal H_{\mathbb{Z}_{p^k}} \cong M_2(\mathbb{Z}_{p^k}).
\]
\end{proposition}
\begin{proof}
Since $\gcd(2, p) = 1$, it follows that $2\in U(\mathbb{Z}_{p^k})$. Hence, it is sufficient to show that the equation
\[
1+\alpha^2+\beta^2=0
\]
can be solved in the ring $\mathbb{Z}_{p^k}$, after which the result follows from Theorem \ref{Matest}.
The proof is carried out by induction on \(k\). This part of the proof follows \cite{Ghosseiri19}.\\
\\
Base case \(k=1\): In this case \(R=\mathbb{Z}_p\) is a finite field of odd order. It is well known that, in every finite field of odd order, the element \(-1\) can be expressed as the sum of two squares (see \cite{Mann}). Hence, there exist \(\alpha,\beta\in \mathbb{Z}_p\) such that \(1+\alpha^2+\beta^2=0\).\\
\\
Induction step: Assume that \(k\ge 2\) and that the statement holds for the ring \(R=\mathbb{Z}_{p^{k-1}}\).\\
\\
Let \(R=\mathbb{Z}_{p^k}\). Its Jacobson radical is \(J(\mathbb{Z}_{p^k})=\langle p\rangle\), and it is nilpotent. Let \(t=k-1\). Then
\[
J(\mathbb{Z}_{p^k})^t = \langle p^{k-1}\rangle  \neq 0 \quad \text{and} \quad J(R)^{t+1} = \langle p^k\rangle = 0.
\]
Consider the quotient ring
\[
S = \mathbb{Z}_{p^k}/J(\mathbb{Z}_{p^k})^t = \mathbb{Z}_{p^k}/\langle p^{k-1}\rangle \cong \mathbb{Z}_{p^{k-1}}.
\]
Since \(S\) is isomorphic to the ring \(\mathbb{Z}_{p^{k-1}}\), the induction hypothesis implies that there exist elements \(\widetilde{\alpha}, \widetilde{\beta}\in S\) such that
\[
\widetilde{1}+\widetilde{\alpha}^2+\widetilde{\beta}^2=0 \quad \text{in the ring } S.
\]
Let \(\alpha,\beta\in \mathbb{Z}_{p^k}\) be representatives of these residue classes. Then the difference belongs to the kernel \(J(\mathbb{Z}_{p^k})^t=\langle p^{k-1}\rangle\), and hence there exists \(w\in \langle p^{k-1}\rangle\) such that
\begin{align}\label{Mash1}
1+\alpha^2+\beta^2 = w.
\end{align}
Since \(w\in \langle p^{k-1}\rangle\) and \(k\ge 2\), it follows that \(2k-2 \ge k\), and therefore
\[
w^2 \in \langle p^{2k-2}\rangle \subseteq \langle p^{k}\rangle = 0.
\]
Hence,
\[
w^2=0.
\]
Equation (\ref{Mash1}) is now considered modulo \(J(\mathbb{Z}_{p^k})=\langle p\rangle\). Since \(w\in \langle p^{k-1}\rangle\subseteq \langle p \rangle\), the following identity is obtained in the quotient field \(\mathbb{Z}_p\):
\[
\widetilde{1}+\widetilde{\alpha}^2+\widetilde{\beta}^2 = 0.
\]
In this field, the elements \(\widetilde{\alpha}\) and \(\widetilde{\beta}\) cannot both be zero, since otherwise \(\widetilde{1}=0\), which is impossible. Hence, at least one of the elements \(\alpha,\beta\) is not divisible by \(p\), that is, at least one is a unit in the ring \(\mathbb{Z}_{p^k}\). Without loss of generality, it may be assumed that
\[
\alpha \in U(\mathbb{Z}_{p^k}).
\]

Since \(p\) is odd, it follows that \(2\in U(\mathbb{Z}_{p^k})\). Hence, \(2\alpha \in U(\mathbb{Z}_{p^k})\), and therefore
\[
\alpha' = \alpha - w(2\alpha)^{-1}
\]
may be defined.
The square \(\alpha'^2\) is computed using the identity \(w^2=0\):
\[
\begin{aligned}
\alpha'^2
&= \left(\alpha - w(2\alpha)^{-1}\right)^2 \\
&= \alpha^2 - 2\alpha w (2\alpha)^{-1} + w^2 (2\alpha)^{-2} \\
&= \alpha^2 - w + 0 \\
&= \alpha^2 - w.
\end{aligned}
\]
Substituting this into equation (\ref{Mash1}) yields
\[
\begin{aligned}
1 + \alpha'^2 + \beta^2
&= 1 + (\alpha^2 - w) + \beta^2 \\
&= (1+\alpha^2+\beta^2) - w \\
&= w - w = 0.
\end{aligned}
\]
Thus, the elements \(\alpha'\) and \(\beta\) satisfy the required equation in the ring \(\mathbb{Z}_{p^k}\). The induction step is complete, and the proof follows.

\end{proof}
Let \(m \in \mathbb{Z}_{p^k}\). Let \(a \in \{0, 1, 2, \ldots, p^k-1\}\) be the integer representative of \(m\). The \textit{$p$-adic valuation} $v_p$ is defined as follows:
\begin{itemize}
    \item If \(m \neq 0\) (that is, \(a \neq 0\)):  
    \(v_p(m)\) is defined to be the largest exponent \(t \in \mathbb{N}_{0}=\{0,1,2,3,\ldots\}\) such that \(p^t\) divides \(a\).
    In other words,
    \[
    v_p(m) = \max\{ t\in \mathbb{N}_{0} : p^t \mid a \}.
    \]
    Since \(0 < a < p^k\), this value always satisfies \(0 \le v_p(m) \le k-1\).

\item If \(m = 0\):  
    It is defined that
    \[
    v_p(0) = k.
    \]
\end{itemize}

\begin{lemma}\label{Padickkaava}
For $0\le t\le k-1$,
\[
\{m\in \mathbb{Z}_{p^k}:v_p(m)=t\}=\langle p^t\rangle \setminus\langle p^{t+1}\rangle
\]
and
\[
|\{m:v_p(m)=t\}|=p^{k-t-1}(p-1).
\]
Moreover, these sets form a pairwise disjoint partition of the ring $\mathbb{Z}_{p^k}$:
\[
\mathbb{Z}_{p^k}=\bigsqcup_{t=0}^k \{m\in \mathbb{Z}_{p^k}:v_p(m)=t\}.
\]
\end{lemma}
\begin{proof}
Assume that
\begin{align*}
m\in\langle p^t\rangle\setminus\langle p^{t+1}\rangle.
\end{align*}
Then $p^t\mid m$, and hence $v_p(m)\ge t$, while $p^{t+1}\nmid m$, and therefore $v_p(m)\le t$. Consequently, $v_p(m)= t$. Thus,
\begin{align}
m\in\langle p^t\rangle\setminus\langle p^{t+1}\rangle
\iff
v_p(m)= t.
\end{align}
This proves the first assertion. The cardinality is computed as
\[
|\{m:v_p(m)=t\}|
=
|\langle p^{t}\rangle|-|\langle p^{t+1}\rangle|
=
p^{k-t}-p^{k-t-1}
=
p^{k-t-1}(p-1),
\]
where formula (\ref{ideaalinkoko}) has been applied.
The third assertion follows from the fact that every \(m\in \mathbb{Z}_{p^k}\) belongs to exactly one set \(\{m:v_p(m)=t\}\) (since \(v_p(m)\) is uniquely determined), and these sets are pairwise disjoint.
\end{proof}

\begin{proposition}\label{Prop9}
If \(p\) is an odd prime and $k\ge 1$, then
\[
|J(\mathcal{H}_{\mathbb{Z}_{p^k}})| =   p^{4k-4}.
\]
\end{proposition}

\begin{proof}
It is proved in \cite[Theorem~3.6]{CheraghpourGhosseiriHeidariZadehSafari} that, for every ring \(R\) satisfying \(2^{-1}\in R\),
\[
J(\mathcal{H}_R)=\mathcal{H}_{J(R)}.
\]
Hence, $J(\mathcal H_{\mathbb{Z}_{p^k}})$ consists of those quaternions whose coefficients belong to $J(\mathbb{Z}_{p^k})$. Since \(|J(\mathbb{Z}_{p^k})|=p^{k-1}\), it follows that
\[
|J(\mathcal{H}_{\mathbb{Z}_{p^k}})|
=
|J(R)|^4
=
(p^{k-1})^4
=
p^{4k-4}.
\]
\end{proof}

\section{The Fibres of the Norm Map}

Next, the cardinalities of the fibers of the norm map
\[
a_{p^k}(m)=\{ q\in \mathbb{Z}_{p^k} : N(q)=m\}
\]
are investigated.
By Propositions \ref{Normi} and \ref{Pkmatriisi},
\[
a_{p^k}(m) =  |\{ A \in M_2(\mathbb{Z}_{p^k}) : \det(A) = m \}|.
\]
Thus we need to count matrices with a given determinant.

\subsection{Counting matrices with unit determinant}
Let
\[
\mathrm{GL}_2(\mathbb{Z}_{p^k})=\{A\in M_2(\mathbb{Z}_{p^k}) : \det(A)\in U(\mathbb{Z}_{p^k})\}.
\]
The following result can be found in \cite{Dolzan21}, but an alternative proof is provided.

\begin{lemma}\label{GLmaara}
Let $p$ be an odd prime and $k\ge 1$. Then
\[
|\mathrm{GL}_2(\mathbb{Z}_{p^k})| = p^{4k-3}(p-1)(p^2-1).
\]
\end{lemma}

\begin{proof}
Propositions \ref{EksJon} and \ref{Aarkund}, together with the isomorphism
$U(\mathcal{H}_{\mathbb{Z}_{p^k}}) \cong \mathrm{GL}_2({\mathbb{Z}_{p^k}})$, yield the exact sequence
\[
1 \longrightarrow 1 + J(\mathcal{H}_{\mathbb{Z}_{p^k}}) \longrightarrow \mathrm{GL}_2({\mathbb{Z}_{p^k}}) \longrightarrow \mathrm{GL}_2(\mathbb{Z}_p) \longrightarrow 1.
\]
Since the sequence is exact, it follows that
\[
|\mathrm{GL}_2({\mathbb{Z}_{p^k}})| = |\mathrm{GL}_2(\mathbb{Z}_p)| \cdot |1 + J(\mathcal{H}_{\mathbb{Z}_{p^k}})|.
\]
In the group $\mathrm{GL}_2(\mathbb{Z}_p)$, the first column of a matrix may be chosen as any nonzero vector of $\mathbb{Z}_p^2$, giving $p^2-1$ possibilities. Once the first column $\vec{v}$ has been chosen, it spans the one-dimensional subspace $\{ \lambda \vec{v} : \lambda\in\mathbb{Z}_p\}$, which contains $p$ elements. Hence, the second column may be chosen in $p^2-p$ ways, and therefore
\[
|\mathrm{GL}_2(\mathbb{Z}_p)| = (p^2-1)(p^2-p) = p(p-1)(p^2-1).
\]
Proposition \ref{Prop9} yields
\[
|J(\mathcal{H}_{\mathbb{Z}_{p^k}})| = p^{4k-4}.
\]
Since
\[
|1 + J(\mathcal{H}_{\mathbb{Z}_{p^k}})| = |J(\mathcal{H}_{\mathbb{Z}_{p^k}})|,
\]
it follows that
\[
|\mathrm{GL}_2({\mathbb{Z}_{p^k}})
= p^{4k-3}(p-1)(p^2-1).
\]
\end{proof}
Next, the cardinality of $a_{p^k}(m)$ is considered for \(m\in U({\mathbb{Z}_{p^k}})\).

\begin{lemma}\label{Tasakuitu}
Each fiber
\[
a_{p^k}(m)
=
\{A\in M_2(\mathbb{Z}_{p^k}) : \det(A)=m\},
\qquad
m\in U(\mathbb{Z}_{p^k}),
\]
contains the same number of elements. In particular, each fiber has the same cardinality as the group
\[
\mathrm{SL}_2(\mathbb{Z}_{p^k})
=
\{A\in \mathrm{GL}_2(\mathbb{Z}_{p^k}) : \det(A)=1\}.
\]
\end{lemma}
\begin{proof}
Consider the determinant map
\[
\det: \mathrm{GL}_2(\mathbb{Z}_{p^k}) \longrightarrow U(\mathbb{Z}_{p^k}).
\]
This map is surjective, since
\[
\det\begin{pmatrix} m & 0 \\ 0 & 1 \end{pmatrix}=m
\]
for every \(m\in U(\mathbb{Z}_{p^k})\). Let \(m\in U(\mathbb{Z}_{p^k})\) be fixed. Left multiplication by the invertible matrix
\[
D_m=\begin{pmatrix} m & 0 \\ 0 & 1 \end{pmatrix}
\]
defines a bijection
\[
\varphi_m: \mathrm{SL}_2(\mathbb{Z}_{p^k}) \longrightarrow {\det}^{-1}(m), \quad \varphi_m(X)=D_mX.
\]
The map is well defined, since
\[
\det(\varphi_m(X))=\det(D_mX)=\det(D_m)\det(X)=m\cdot 1=m.
\]
Its inverse is given by \(Y\mapsto D_m^{-1}Y\), and hence it is a bijection. Therefore,
\[
|{\det}^{-1}(m)|=|\mathrm{SL}_2(\mathbb{Z}_{p^k})|
\]
for every \(m\in U(\mathbb{Z}_{p^k})\). Thus, all fibers of the determinant map have the same cardinality.
\end{proof}
 
The preceding proof is valid for any finite commutative ring $R$, not only for $\mathbb{Z}_{p^k}$. The only requirement is that $m\in U(R)$ and that the matrix $D_m$ be invertible, which holds because $m$ is a unit.

\begin{proposition}\label{Uksikkokuidutmaara}
If \(m\in U({\mathbb{Z}_{p^k}})\), then
\[
a_{p^k}(m)=p^{3k-2}(p^2-1).
\]
\end{proposition}

\begin{proof}
By the previous lemma,
\[
a_{p^k}(m)=|\mathrm{SL}_2(\mathbb{Z}_{p^k})|.
\]
The determinant induces a surjective group homomorphism
\[
\det:\mathrm{GL}_2({\mathbb{Z}_{p^k}})\longrightarrow U({\mathbb{Z}_{p^k}}),
\]
whose kernel is
\[
\mathrm{ker(det)}=\mathrm{SL}_2(\mathbb{Z}_{p^k})=\{A\in \mathrm{GL}_2(\mathbb{Z}_{p^k}) : \det(A)=1\}.
\]
By the First Isomorphism Theorem,
\[
\mathrm{GL}_2(\mathbb{Z}_{p^k})/\mathrm{SL}_2(\mathbb{Z}_{p^k}) \cong\mathrm{im(det)}=U(\mathbb{Z}_{p^k}),
\]
and hence
\[
|\mathrm{SL}_2(\mathbb{Z}_{p^k})|
=
\frac{|\mathrm{GL}_2(\mathbb{Z}_{p^k})|}{|U(\mathbb{Z}_{p^k})|}.
\]
Lemma \ref{GLmaara} gives
\[
|\mathrm{GL}_2(\mathbb{Z}_{p^k})|
=
p^{4k-3}(p-1)(p^2-1),
\]
and Proposition \ref{YksRom} gives
\[
|U(\mathbb{Z}_{p^k})|
=
p^{k-1}(p-1),
\]
thus
\[
a_{p^k}(m)
=
\frac{p^{4k-3}(p-1)(p^2-1)}
{p^{k-1}(p-1)}
=
p^{3k-2}(p^2-1).
\]
\end{proof}

\subsection{Counting matrices with nonunit determinant}
Next, it is proved that the representation number $a_{p^k}(m)$ depends only on the \(p\)-adic valuation of \(m\).

\begin{theorem}\label{TasJakLause}
If \(m_1,m_2\in \mathbb Z_{p^k}\) satisfy
\[
v_p(m_1)=v_p(m_2),
\]
then
\[
a_{p^k}(m_1)=a_{p^k}(m_2).
\]
\end{theorem}

\begin{proof}
Assume that \(t=v_p(m_1)=v_p(m_2)\), that is, $m_1=ap^t$ and $m_2=bp^t$, where $a$ and $b$ are not divisible by $p$, or equivalently, $a,b\in U(\mathbb Z_{p^k})$. Then there exists a unit
\(u=ba^{-1}\in U(\mathbb Z_{p^k})\) such that
\[
m_2=um_1.
\]
Let
\[
D_u=
\begin{pmatrix}
u&0\\
0&1
\end{pmatrix}.
\]
The mapping
\[
A\longmapsto D_uA
\]
is a bijection from \(\det^{-1}(m_1)\) onto
\(\det^{-1}(m_2)\), since
\[
\det(D_uA)
=\det(D_u)\det(A)
=u\,m_1
=m_2.
\]
Therefore, these fibers have the same cardinality.
\end{proof}

\subsection{Complement Formula for the Zero Fiber}

Let \(0 \le t \le k-1\). Define
\[
A_t := a_{p^k}(m),
\]
where \(m \in \mathbb{Z}_{p^k}\) satisfies \(v_p(m)=t\). This is well defined by Theorem \ref{TasJakLause}. The following lemma is now proved. It states that, if the cardinalities $A_t$ of all nonzero determinant fibers are known, then the cardinality of the zero fiber $a_{p^k}(0)$ can be computed by subtracting them from the total number of matrices.

\begin{lemma}\label{NollakuidunKkaava}
With the above notation,
\[
a_{p^k}(0)
=
p^{4k}
-
\sum_{t=0}^{k-1}
p^{k-t-1}(p-1)\, A_t.
\]
\end{lemma}

\begin{proof}
The fibers of the determinant map
\[
\det: M_2(\mathbb Z_{p^k}) \to \mathbb Z_{p^k}
\]
form a partition of the space \(M_2(\mathbb Z_{p^k})\) into the sets \(\det^{-1}(m)\), where \(m \in \mathbb Z_{p^k}\). Let $0\le t\le k-1$.
Lemma \ref{Padickkaava} states that
\[
|\{m:v_p(m)=t\}|
=
p^{k-t-1}(p-1).
\]
By Theorem \ref{TasJakLause}, the fibers corresponding to all such elements \(m\) have the same cardinality, namely \(A_t\). The element \(0\) does not belong to any of the preceding sets, since $v_p(0)=k$. Assume that the fiber over \(0\) contains \(a_{p^k}(0)\) elements. Therefore,
\[
|M_2(\mathbb Z_{p^k})|
=
a_{p^k}(0)
+
\sum_{t=0}^{k-1}
p^{k-t-1}(p-1)\, A_t,
\]
and since
\[
|M_2(\mathbb Z_{p^k})|=|\mathbb Z_{p^k}|^4=p^{4k},
\]
the desired result is obtained.
\end{proof}


\subsection{Main Theorem}

\begin{theorem}\label{paalause}
Let $p$ be an odd prime, $k\ge 1$, and $m\in \mathbb Z_{p^k}$. Let $t=v_p(m)$. Then the cardinality of the determinant fiber
\[
a_{p^k}(m)=\left|\{A\in M_2(\mathbb Z_{p^k}):\det(A)=m\}\right|
\]
is given by
\[ a_{p^k}(m)= \begin{cases} p^{3k-2}(p^2-1), & t=0,\\
p^{3k-2-t}(p+1)(p^{t+1}-1), & 0<t<k,\\
p^{2k-1}(p^{k+1}+p^k-1), & t=k. 
\end{cases} 
\]
\end{theorem}
The first case has already been established in Proposition \ref{Uksikkokuidutmaara}.
The remaining cases are proved in the following section; see formulas (\ref{Kaavai<j})
and (\ref{Tapausk=Nolla}). \\
\\
The case $k=1$ is of particular interest, since $\mathbb{Z}_p$ is a finite field, and therefore every nonzero element is invertible. In this case, an element $m\in\mathbb{Z}_p$ has only the two boundary cases: $t=v_p(m)=0$ when $m$ is invertible, and $t=v_p(0)=1$ when $m=0$. The fiber cardinalities are then
\[
a_{p^k}(m)=
\begin{cases}
p(p^2-1), & m\neq 0,\\
p(p^{2}+p-1), & m=0.
\end{cases}
\]
This also shows that the $p$-adic valuation values $0<t<k$ occur only in the rings $\mathbb{Z}_{p^k}$ with $k\ge 2$. This is due to the existence of nonzero noninvertible elements in these rings.

\section{Proof of the Main Theorem}
The proof is based on the Smith normal form, which partitions the matrices into small, manageable classes. Within each class, the \(p\)-adic valuation of the determinant is constant, while the unit part of the determinant is uniformly distributed. Consequently, the cardinality of each determinant fiber is obtained by summing the cardinalities of the appropriate classes and dividing by the number of units.

\subsection{Smith Normal Form}

Since \(  \mathbb{Z}_{p^k}\) is a principal ideal ring, every matrix \(A \in M_2(\mathbb{Z}_{p^k})\) can be written in Smith normal form (see formula (\ref{NULJOK}) and \cite[Theorem 3.2]{Kaplansky}),
\[
A = U D_{i,j} V^{-1},
\]
where \(U,V \in \mathrm{GL}_2(\mathbb{Z}_{p^k})\) and
\[
D_{i,j} = \begin{pmatrix} p^i & 0 \\ 0 & p^j \end{pmatrix}, \qquad 0 \le i \le j \le k.
\]
The Smith normal form is unique: the exponents \((i,j)\) are uniquely determined by the condition \(0\le i\le j\le k\). The Smith classes
\[
\mathcal O(D_{i,j}) = \{\, U D_{i,j} V^{-1} \mid U,V \in \mathrm{GL}_2(\mathbb{Z}_{p^k}) \,\}
\]
are precisely the \(\mathrm{GL}_2(\mathbb{Z}_{p^k})\times \mathrm{GL}_2(\mathbb{Z}_{p^k})\)-orbits under the action
\[
(U,V)\cdot A = U A V^{-1}.
\]
Thus, the group \(\mathrm{GL}_2(\mathbb{Z}_{p^k}) \times \mathrm{GL}_2(\mathbb{Z}_{p^k})\) acts on \(M_2(\mathbb{Z}_{p^k})\) by left and right multiplication. The Smith normal form classes are the orbits of this action. The orbit cardinality is given by the orbit--stabilizer formula:
\[
|\mathcal O(D_{i,j})| = \frac{|\mathrm{GL}_2(\mathbb{Z}_{p^k})|^2}{|\operatorname{Stab}(D_{i,j})|},
\]
where
\[
\operatorname{Stab}(D_{i,j}) = \{ (U,V) \in \mathrm{GL}_2(\mathbb{Z}_{p^k})^2 \mid U D_{i,j} V^{-1} = D_{i,j} \}
= \{ (U,V) \in \mathrm{GL}_2(\mathbb{Z}_{p^k})^2 \mid U D_{i,j} = D_{i,j} V \}.
\]
The valuation of the determinant is determined as follows. If
\[
A = U D_{i,j} V^{-1},
\]
then
\[
\det(A) = \det(U)\det(D_{i,j})\det(V)^{-1}
= p^{i+j}\frac{\det(U)}{\det(V)}.
\]
Since \(\det(U), \det(V) \in U(\mathbb{Z}_{p^k})\), it follows that
\[
u:=\frac{\det(U)}{\det(V)} \in U(\mathbb{Z}_{p^k}).
\]
Hence, by (\ref{potemzo}),
\[
\det(A)=
\begin{cases}
p^{i+j}u, & \text{if } 0\le i+j<k,\\
0, & \text{if } i+j\ge k.
\end{cases}
\]
Thus, the \(p\)-adic valuation of the determinant is determined solely by the sum \(i+j\). Every Smith class with \(i+j=t<k\) consists of matrices whose determinants have \(p\)-adic valuation \(t\). Consequently, the set
\[
\{A\in M_2(\mathbb{Z}_{p^k}) : v_p(\det A)=t\}
=
\bigsqcup_{i+j=t}\mathcal O(D_{i,j})
\]
is the disjoint union of the corresponding Smith classes. Similarly, the zero fiber
\[
{\det}^{-1}(0)
=
\bigsqcup_{i+j\ge k}\mathcal O(D_{i,j})
\]
is the disjoint union of those Smith classes for which \(i+j\ge k\).\\

Since \(\mathbb{Z}_{p^k}\) is a commutative principal ideal ring, the Smith normal form is unique (see \cite[Theorem 9.3]{Kaplansky}). Therefore, the Smith classes form a partition of the matrix ring \(M_2(\mathbb{Z}_{p^k})\):
\begin{align}\label{orbitositus}
M_2(\mathbb Z_{p^k})
=
\bigsqcup_{0\le i\le j\le k}
\mathcal O(D_{i,j}).
\end{align}
Moreover, these orbits satisfy several natural properties, which are summarized in the following proposition.

\begin{proposition}[Classification of Orbits]
Let
\[
D_{i,j}=\begin{pmatrix}p^i&0\\0&p^j\end{pmatrix},
\qquad
0\le i\le j\le k.
\]
Then
\begin{enumerate}
    \item Every matrix \(A\in M_2(\mathbb Z_{p^k})\) belongs to exactly one orbit

   \[
   \mathcal O(D_{i,j}).
   \]
\item If

   \[
\mathcal O(D_{i,j})=\mathcal O(D_{i',j'}),
   \]

   then

   \[
   (i,j)=(i',j').
   \]
\item In particular,

   \[
   \mathcal O(D_{i,j})\cap
   \mathcal O(D_{i',j'})
   =\varnothing,
   \qquad
   (i,j)\neq(i',j').
   \]
\end{enumerate}
   
\end{proposition}
\begin{proof}
This is an immediate consequence of the Smith normal form over the ring
\(\mathbb Z_{p^k}\): every matrix is left--right equivalent to a unique diagonal form
\[
\begin{pmatrix}p^i&0\\0&p^j\end{pmatrix},
\qquad 0\le i\le j\le k.
\]
The uniqueness of the Smith normal form proves part (2), and
part (3) follows immediately from parts (1)--(2).
\end{proof}

\subsection{Sizes of Stabilizers}

We now consider the stabilizer
\[
\operatorname{Stab}(D_{i,j})=\{(U,V): UD_{i,j}=D_{i,j}V\},
\]
where
\[
U = \begin{pmatrix} a & b \\ c & d \end{pmatrix}, \qquad
V = \begin{pmatrix} e & f \\ g & h \end{pmatrix}.
\]
The equation \(U D_{i,j} = D_{i,j} V\) gives rise to congruence conditions that depend on \(i\), \(j\), and \(k\). These conditions can be solved systematically.

\begin{lemma}\label{StabKoko1}
Let \(p\) be an odd prime, \(k\ge 1\), and let \(0 \le i < j < k\). Define
\[
D_{i,j} = \begin{pmatrix} p^i & 0 \\ 0 & p^j \end{pmatrix}.
\]
Then
\[
|\operatorname{Stab}(D_{i,j})| = p^{4k+3i+j-2}(p-1)^2.
\]
\end{lemma}
\begin{proof} 
Write
\[
U = \begin{pmatrix} a & b \\ c & d \end{pmatrix}, \qquad
V = \begin{pmatrix} e & f \\ g & h \end{pmatrix},
\]
where $U,V\in M_2(\mathbb{Z}_{p^k})$.
The equation \(U D_{i,j} = D_{i,j} V\) is
\[
\begin{pmatrix}
a p^i & b p^j \\
c p^i & d p^j
\end{pmatrix}
=
\begin{pmatrix}
p^i e & p^i f \\
p^j g & p^j h
\end{pmatrix},
\]
which is equivalent to the following four linear equations:
\[
p^i e = p^i a, \quad p^i f = p^j b, \quad p^j g = p^i c, \quad p^j h = p^j d,
\]
over the ring $\mathbb{Z}_{p^k}$,
or equivalently (since $i<j$),
\begin{align*}
    e\equiv&\ a \mod p^{k-i}\\
    f\equiv&\ bp^{j-i} \mod p^{k-i}\\
    c\equiv&\ gp^{j-i} \mod p^{k-i}\\
    h\equiv&\ d \mod p^{k-j}.
\end{align*}
Define
\begin{align}\label{Tparam}
T_{ij}:M_2(\mathbb{Z}_{p^k})\times M_2(\mathbb{Z}_{p^k})\to M_2(\mathbb{Z}_{p^k});\ T(U,V)=U D_{i,j} - D_{i,j} V.
\end{align}
Then
\[
\ker(T_{ij})=\{(U,V): T_{ij}\equiv 0 \mod p^k\}=\{(U,V):U D_{i,j} = D_{i,j} V\}.
\]
We solve the equations explicitly:

\begin{align*}
p^i e = p^i a &\iff e-a \in  \langle p^{k-i}\rangle\ \iff e = a + p^{k-i}\alpha, \quad \alpha \in \mathbb Z_{p^i},\\
p^i f = p^j b &\iff f - p^{j-i}b \in \langle p^{k-i}\rangle \iff f = p^{j-i}b + p^{k-i}\beta, \quad \beta \in \mathbb Z_{p^i},\\
p^j g = p^i c &\iff c - p^{j-i}g \in  \langle p^{k-i}\rangle \iff c = p^{j-i}g + p^{k-i}\gamma, \quad \gamma \in \mathbb Z_{p^i},\\
p^j h = p^j d &\iff h-d \in  \langle p^{k-j}\rangle \iff h = d + p^{k-j}\eta, \quad \eta \in \mathbb Z_{p^j}.
\end{align*}
Let
\[
T_{ij}(U,V)=\begin{pmatrix}
a p^i-p^i e & b p^j-p^i f \\
c p^i-p^j g & d p^j-p^j h
\end{pmatrix}
=
\begin{pmatrix}
x_{1,1} & x_{1,2} \\
x_{2,1} & x_{2,2}
\end{pmatrix}.
\]
Here
\begin{itemize}
    \item $x_{1,1}= p^i(a-e)\in \langle p^i\rangle$,
    \item $x_{1,2}=p^{i}(bp^{j-i} - f)\in \langle p^i\rangle$,
    \item $x_{2,1}=p^i(c-gp^{j-i})\in \langle p^i\rangle$,
    \item $x_{2,2}= p^j(d- h)\in \langle p^j\rangle$.   
\end{itemize}
Hence
\[
\operatorname{Im}(T_{ij}) = 
\begin{pmatrix}
\langle p^i \rangle & \langle p^i \rangle \\
\langle p^i \rangle & \langle p^j \rangle
\end{pmatrix}.
\]
Observe that
\begin{align*}
    T_{ij}((U_1,V_1)+(U_2,V_2))&=T_{ij}(U_1+U_2,V_1+V_2)\\
    &=(U_1+U_2)D_{ij}-D_{ij}(V_1+V_2)\\
    &=U_1D_{ij}-D_{ij}V_1+U_2D_{ij}-D_{ij}V_2\\
    &=T_{ij}(U_1,V_1)+T_{ij}(U_2,V_2)
\end{align*}
and similarly,
\[
T_{ij}(r(U,V))=rT_{ij}(U,V),\qquad r\in\mathbb{Z}_{p^k},
\]
so \(T_{ij}\) is a linear map between modules.
The first isomorphism theorem yields
\[
 \big(M_2(\mathbb{Z}_{p^k})\times M_2(\mathbb{Z}_{p^k})\big)/\ker(T_{ij})\cong \operatorname{Im}(T_{ij}).
\]
Since
\[
|M_2(\mathbb{Z}_{p^k})|=p^{4k}
\]
and, by (\ref{ideaalinkoko}),
\[
|\operatorname{Im}(T_{ij})|=(p^{k-i})^3 p^{k-j}=p^{4k-3i-j},
\]
the size of the kernel is
\[
|\ker(T_{ij})|=\frac{|M_2(\mathbb{Z}_{p^k})|^2}{|\operatorname{Im}(T_{ij})|}=\frac{p^{8k}} {p^{4k-3i-j}}=p^{4k+3i+j}.
\]
For verification, we also compute the size of the kernel directly. If $(U,V)\in \ker(T_{ij})$, then
\[
U = \begin{pmatrix} a & b \\ p^{j-i}g + p^{k-i}\gamma & d \end{pmatrix}, \qquad
V = \begin{pmatrix} a + p^{k-i}\alpha & p^{j-i}b + p^{k-i}\beta \\ g & d + p^{k-j}\eta \end{pmatrix},
\]
where $\alpha, \beta, \gamma\in \mathbb Z_{p^i}$ and $\eta \in \mathbb Z_{p^j}$.
The parameters can be chosen as follows:
\begin{itemize}
    \item $a,b,d,g\in\mathbb{Z}_{p^k}$, each with $p^k$ choices,
    \item $\alpha, \beta, \gamma\in \mathbb Z_{p^i}$, each with $p^{i}$ choices,
    \item $\eta \in \mathbb Z_{p^j}$, with $p^{j}$ choices.
\end{itemize}
Hence
\[
|\ker(T_{ij})|=(p^k)^4\cdot (p^{i})^3\cdot p^{j}=p^{4k+3i+j}.
\]
The invertible pairs $(U,V)$ are precisely
\[
(\mathrm{GL}_2(\mathbb{Z}_{p^k})\times \mathrm{GL}_2(\mathbb{Z}_{p^k}))\cap \ker(T_{ij}).
\]
By Proposition \ref{YksRom}, $a\in U(\mathbb{Z}_{p^k})$ if and only if $\gcd(a, p) = 1$, which is equivalent to $a\not\equiv 0 \mod p$.
Accordingly, the matrices must satisfy the additional conditions below. First observe that
\begin{align*}
   \det(U)&= ad-b(p^{j-i}g + p^{k-i}\gamma)\equiv ad  \pmod p,\\
   \det(V)&=(a + p^{k-i}\alpha)(d + p^{k-j}\eta)-g( p^{j-i}b + p^{k-i}\beta) \equiv ad  \pmod p.
\end{align*}
Now
 $\det(U),\det(V)\in U(\mathbb{Z}_{p^k})$ if and only if
\[
ad\not\equiv 0  \pmod p.
\]
This is equivalent to
\[
a\not\equiv 0  \pmod p\quad \text{and}\quad d\not\equiv 0  \pmod p.
\]
Thus the remaining parameters $b,g,\alpha,\beta,\gamma,\eta$ may be chosen arbitrarily, and, as shown above, there are altogether $(p^k)^2\cdot (p^i)^3\cdot p^j$ such choices.\\
\\
It remains to count the choices for $a$ and $d$. By Corollary \ref{nojakkaaa},
$a,d\notin \langle p\rangle$, and by (\ref{ideaalinkoko}) each can be chosen in
 \[
 |\mathbb{Z}_{p^k}\setminus \langle p\rangle |=|\mathbb{Z}_{p^k}|-|\langle p\rangle|=p^k-p^{k-1}=p^{k-1}(p-1)
 \]
ways. Consequently,
\[
|(\mathrm{GL}_2(\mathbb{Z}_{p^k})\times \mathrm{GL}_2(\mathbb{Z}_{p^k}))\cap \ker(T_{ij})|= p^{4k+3i+j-2}(p-1)^2.
\]
\end{proof}

\begin{lemma}\label{StabKoko2}
Let \(p\) be an odd prime, \(k\ge 1\), and let \(0 \le i < k\). Define
\[
D_{i,k} = \begin{pmatrix} p^i & 0 \\ 0 & 0 \end{pmatrix}.
\]
Then
\[
|\operatorname{Stab}(D_{i,k})|=p^{5k+3i-3}(p-1)^3.
\]
\end{lemma}
\begin{proof}
Write
\[
U = \begin{pmatrix} a & b \\ c & d \end{pmatrix}, \qquad
V = \begin{pmatrix} e & f \\ g & h \end{pmatrix},
\]
where $U,V\in M_2(\mathbb{Z}_{p^k})$.
The equation \(U D_{i,k} = D_{i,k} V\) is
\[
\begin{pmatrix}
a p^i & 0 \\
c p^i & 0
\end{pmatrix}
=
\begin{pmatrix}
p^i e & p^i f \\
0 & 0
\end{pmatrix},
\]
We obtain
\[
a p^i=p^ie,\quad  c p^i=0,\quad p^i f=0,
\]
and, for the time being, no restrictions are imposed on the parameters $d,h\in \mathbb{Z}_{p^k}$. The solutions are
\begin{align*}
   &a p^i=p^ie\iff a-e\in\langle p^{k-i}\rangle \iff e=a+p^{k-i}\alpha,\ \alpha\in\mathbb{Z}_{p^i}, \\
   &c p^i=0 \iff c\in \langle p^{k-i} \rangle \iff c=\beta p^{k-i},\ \beta\in \mathbb{Z}_{p^i},\\
   &p^i f=0\iff  f\in \langle  p^{k-i}\rangle \iff f=\gamma p^{k-i},\ \gamma \in\mathbb{Z}_{p^i}.
\end{align*}
Define $T_{ik}$ as in (\ref{Tparam}), that is,
\[
T_{ik}(U,V)=\begin{pmatrix}
a p^i-p^i e & -p^i f \\
c p^i & 0
\end{pmatrix}
=
\begin{pmatrix}
x_{1,1} & x_{1,2} \\
x_{2,1} & x_{2,2}
\end{pmatrix}.
\]
Observe that
\begin{itemize}
    \item $x_{1,1}=p^i(a-e)\in\langle p^{i}\rangle$,
    \item $x_{1,2}=-p^i f\in \langle p^i\rangle$,
    \item $x_{2,1}=c p^i\in \langle p^i\rangle$,
    \item $x_{2,2}=0\in \langle p^k\rangle=\{0\}$,
\end{itemize}
that is,
\[
\operatorname{Im}(T_{ik})=
\begin{pmatrix}
\langle p^i \rangle & \langle p^i \rangle \\
\langle p^i \rangle & \{0\}
\end{pmatrix}.
\]
Hence, by (\ref{ideaalinkoko}),
\[
|\operatorname{Im}(T_{ik})|=(p^{k-i})^3=p^{3k-3i},
\]
and the kernel has cardinality
\[
|\ker(T_{ik})|=\frac{|M_2(\mathbb{Z}_{p^k})|^2}{|\operatorname{Im}(T_{ik})|}=\frac{p^{8k}}{p^{3k-3i}}=p^{5k+3i}.
\]
We also compute this in a another way. If $(U,V)\in \ker(T_{ik})$, then
\[
U = \begin{pmatrix} a & b \\
\beta p^{k-i} & d \end{pmatrix}, \qquad
V = \begin{pmatrix} a+\alpha p^{k-i} &  \gamma p^{k-i} \\
g & h \end{pmatrix},
\]
where $a,b,d,g,h\in\mathbb{Z}_{p^k}$ and $\alpha,\beta,\gamma\in \mathbb{Z}_{p^i}$. Hence there are altogether $(p^k)^5\cdot (p^i)^3=p^{5k+3i}$ possible choices. Next, we compute the cardinality of the set of invertible matrices
\[
(\mathrm{GL}_2(\mathbb{Z}_{p^k})\times \mathrm{GL}_2(\mathbb{Z}_{p^k}))\cap \ker(T_{ik}).
\]
As in the previous lemma, we obtain the additional conditions
\begin{align*}
    \det(U)&=ad-b\beta p^{k-i} \equiv ad \mod p,\\
    \det(V)&=(a+\alpha p^{k-i})h-g\gamma p^{k-i}\equiv ah \mod p.
\end{align*}
By the same argument as in the previous lemma, it follows that
\begin{itemize}
    \item $a,d,h\not\equiv 0 \mod p$, and each of these can be chosen in $p^{k-1}(p-1)$ ways,
    \item $\alpha,\beta,\gamma\in\mathbb{Z}_{p^i}$ can each be chosen in $p^i$ ways,
    \item $b,g\in\mathbb{Z}_{p^k}$ can each be chosen in $p^k$ ways.
\end{itemize}
Hence
\[
|(\mathrm{GL}_2(\mathbb{Z}_{p^k})\times \mathrm{GL}_2(\mathbb{Z}_{p^k}))\cap \ker(T_{ik})|
=\big(p^{k-1}(p-1)\big)^3 (p^i)^3 (p^k)^2
=p^{5k+3i-3}(p-1)^3.
\]
\end{proof}

\begin{lemma}\label{StabKoko3}
Let \(0 \le i < k\) and
\[
D_{i,i} = \begin{pmatrix} p^i & 0 \\ 0 & p^i \end{pmatrix}.
\]
Then
\[
|\operatorname{Stab}(D_{i,i})| =   p^{4k+4i-3}(p-1)(p^2-1).
\]
\end{lemma}

\begin{proof}
Since \(D_{i,i}=p^i I_2\), the equation \(U D_{i,i} = D_{i,i} V\) is equivalent to
\[
p^i U = p^i V,
\]
that is,
\[
p^i (U - V) = 0.
\]
Since, in the ring \( \mathbb{Z}_{p^k}\), the equality \(p^i X = 0\) holds if and only if \(X \in M_2(\langle p^{k-i}\rangle )\), it follows that
\[
V = U + Y,
\qquad
Y \in M_2(\langle p^{k-i}\rangle).
\]
For \(2\times 2\) matrices over a commutative ring, the following determinant identity holds:
\[
\det(U+Y)
=
\det(U)
+
\det(Y)
+
\operatorname{tr}(\operatorname{adj}(U)Y).
\]
From this it follows that
\[
\det(U+Y)\equiv \det(U) \mod p,
\]
that is,
\[
\det(U+Y)\in U(\mathbb{Z}_{p^k}) \iff \det(U)\in U(\mathbb{Z}_{p^k}),
\]
and hence, in summary,
\[
U\in \mathrm{GL}_2(\mathbb{Z}_{p^k})\iff U+Y\in \mathrm{GL}_2(\mathbb{Z}_{p^k}).
\]
Define the mapping
\[
\Phi: \mathrm{GL}_2(\mathbb{Z}_{p^k}) \times M_2(\langle p^{k-i}\rangle)
\longrightarrow
\operatorname{Stab}(D_{i,i}),
\qquad
(U,Y) \longmapsto (U, U+Y).
\]
This mapping is a bijection:
\begin{itemize}
    \item Injectivity: If \(\Phi(U,Y) = \Phi(U',Y')\), then \(U = U'\) and \(U+Y = U'+Y'\), so \(Y = Y'\).
    \item Surjectivity: If \((U,V) \in \operatorname{Stab}(D_{i,i})\), then \(V = U + Y\) for some \(Y \in M_2(\langle p^{k-i}\rangle)\). Hence \(\Phi(U, V-U) = (U,V)\).
\end{itemize}
Therefore,
\[
|\operatorname{Stab}(D_{i,i})|
=
|\mathrm{GL}_2(\mathbb{Z}_{p^k})| \; |M_2(\langle p^{k-i}\rangle)|,
\]
where
\[
|M_2(\langle p^{k-i}\rangle)| = |\langle p^{k-i}\rangle|^4 = p^{4i},
\]
and, by Lemma \ref{GLmaara},
\[
|\mathrm{GL}_2(\mathbb{Z}_{p^k})| = p^{4k-3}(p-1)(p^2-1).
\]
Thus,
\[
|\operatorname{Stab}(D_{i,i})|=p^{4k+4i-3}(p-1)(p^2-1).
\]
\end{proof}

\begin{remark}
The case \(i = k\) is exceptional. In this case \(D_{k,k} = 0\), so every pair \((U,V) \in \mathrm{GL}_2(\mathbb{Z}_{p^k}) \times \mathrm{GL}_2(\mathbb{Z}_{p^k})\) satisfies the condition \(U D = D V\). Hence
\[
|\operatorname{Stab}(D_{k,k})| = |\mathrm{GL}_2(\mathbb{Z}_{p^k})|^2.
\]
\end{remark}

\subsection{Orbit Sizes}

The following formulas give the orbit sizes in the different cases.

\begin{proposition}\label{ORBKOK}
Let $p$ be an odd prime and $k\ge 1$.
Let $D_{i,j}$ be as above, with $0\le i\le j\le k$.
Then
\begin{align*}
|\mathcal O(D_{i,i})| &= p^{4k-4i-3}(p-1)^2(p+1), &&0\le i<k,\\
|\mathcal O(D_{i,j})| &= p^{4k-3i-j-4}(p-1)^2(p+1)^2, &&0\le i<j<k,\\
|\mathcal O(D_{i,k})| &= p^{3k-3i-3}(p-1)(p+1)^2, &&i<k,
\end{align*}
and
\[
|\mathcal O(D_{k,k})|=1.
\]
\end{proposition}
\begin{proof}
By the orbit--stabilizer theorem,
\[
|\mathcal O(D_{i,j})|
=
\frac{|\mathrm{GL}_2(\mathbb Z_{p^k})|^2}
{|\operatorname{Stab}(D_{i,j})|},
\]
where
\[
|\mathrm{GL}_2(\mathbb Z_{p^k})|
=
p^{4k-3}(p-1)^2(p+1).
\]
If \(i=j<k\), then, by Lemma \ref{StabKoko3},
\(|\operatorname{Stab}(D_{i,i})|=p^{4k+4i-3}(p-1)(p^2-1)\), and hence
\[
|\mathcal O(D_{i,i})|
=
\frac{|\mathrm{GL}_2(\mathbb Z_{p^k})|^2}
{|\mathrm{GL}_2(\mathbb Z_{p^k})|\,p^{4i}}
=
p^{\,4k-3-4i}(p-1)^2(p+1).
\]
If \(i<j<k\), then, by Lemma \ref{StabKoko1},
\(|\operatorname{Stab}(D_{i,j})|=p^{\,4k+3i+j-2}(p-1)^2\), and therefore
\[
|\mathcal O(D_{i,j})|
=
\frac{\left(p^{4k-3}(p-1)^2(p+1)\right)^2}
{p^{\,4k+3i+j-2}(p-1)^2}
=
p^{\,4k-4-3i-j}(p-1)^2(p+1)^2.
\]
If \(i<k\) and \(j=k\), then, by Lemma \ref{StabKoko2},
\(|\operatorname{Stab}(D_{i,k})|=p^{\,5k+3i-3}(p-1)^3\). Hence
\[
|\mathcal O(D_{i,k})|
=
\frac{p^{\,8k-6}(p-1)^4(p+1)^2}
{p^{\,5k+3i-3}(p-1)^3}
=
p^{\,3k-3i-3}(p-1)(p+1)^2.
\]
Finally, since \(D_{k,k}=0\), its orbit consists only of the zero matrix. Therefore,
\[
|\mathcal O(D_{k,k})|=1.
\]
\end{proof}
Since, by (\ref{orbitositus}), the orbits form a partition of the matrix space $M_2(\mathbb{Z}_{p^k})$, the calculations in the previous proposition can be verified by summing the orbit sizes. The resulting sum must be equal to
\[
| M_2(\mathbb{Z}_{p^k})|=p^{4k}.
\]

\begin{lemma}[Sum of the Orbits]
\[
\sum_{i,j=0}^k |\mathcal O(D_{i,j})|=p^{4k}.
\]

\end{lemma}
\begin{proof}
\begin{align*}
    \sum_{i,j=0}^k |\mathcal O(D_{i,j})|=&\sum_{0\le i<j<k} |\mathcal O(D_{i,j})|+\sum_{i=0}^{k-1}|\mathcal O(D_{i,i})|\\
    &+\sum_{i=0}^{k-1}|\mathcal O(D_{i,k})|+|\mathcal O(D_{k,k})|\\
    =&\sum_{0\le i<j<k} p^{4k-3i-j-4}(p-1)^2(p+1)^2+\sum_{i=0}^{k-1}p^{4k-4i-3}(p-1)^2(p+1)\\
    &+\sum_{i=0}^{k-1}p^{3k-3i-3}(p-1)(p+1)^2+1\\
    =&p^{4k-4}(p-1)^2(p+1)^2\sum_{0\le i<j<k}p^{-3i-j} +p^{4k-3}(p-1)^2(p+1)\sum_{i=0}^{k-1}p^{-4i}\\
    &+p^{3k-3}(p-1)(p+1)^2\sum_{i=0}^{k-1}p^{-3i}+1.
\end{align*}
The geometric sums are computed as follows:
\begin{align*}
    \sum_{i=0}^{k-1}p^{-4i}&= \frac{p^4 - p^{4-4k}}{p^4 - 1},\\
    \sum_{i=0}^{k-1}p^{-3i}&=\frac{p^3 - p^{3-3k}}{p^3 - 1},
\end{align*}
and
\[
\sum_{0\le i<j<k}p^{-3i-j}=\frac{p^3}{(p-1)(p^2+p+1)}\left(\frac{p^{k-1}-1}{p^{k-1}(p-1)}-\frac{p^{4k-4}-1}{p^{4k-4}(p^4-1)}\right)
\]
It is obtained that
\begin{align*}
    \sum_{i,j=0}^k |\mathcal O(D_{i,j})| 
    =&p^{4k-4}(p-1)^2(p+1)^2\frac{p^3}{(p-1)(p^2+p+1)}\left(\frac{p^{k-1}-1}{p^{k-1}(p-1)}-\frac{p^{4k-4}-1}{p^{4k-4}(p^4-1)}\right)\\
    &+p^{4k-3}(p-1)^2(p+1)\frac{p^4 - p^{4-4k}}{p^4 - 1}
    +p^{3k-3}(p-1)(p+1)^2\frac{p^3 - p^{3-3k}}{p^3 - 1}+1\\
    =& (p+1)^2\frac{p^3}{(p^2+p+1)}p^{3k-3}(p^{k-1}-1) -(p-1)(p+1)^2\frac{p^3}{(p^2+p+1)}\frac{p^{4k-4}-1}{p^4-1}\\
    &+p^{4k-3}(p-1)^2(p+1)\frac{p^{4k} - 1}{p^{4k-4}(p^4 - 1)}
    +p^{3k-3}(p-1)(p+1)^2\frac{p^{3k} - 1}{p^{3k-3}(p^3 - 1)}+1\\
    &=p^{4k}S,
\end{align*}
where
\begin{align*}
    S=& (p+1)^2\frac{p^3}{(p^2+p+1)}p^{-k-3}(p^{k-1}-1) -(p-1)(p+1)^2\frac{p^{3-4k}}{(p^2+p+1)}\frac{p^{4k-4}-1}{p^4-1}\\
    &+p^{-3}(p-1)^2(p+1)\frac{p^{4k} - 1}{p^{4k-4}(p^4 - 1)}
    +p^{-k-3}(p-1)(p+1)^2\frac{p^{3k} - 1}{p^{3k-3}(p^3 - 1)}+p^{-4k}
\end{align*}
Since $p^3-1=(p-1)(p^2+p+1)$, the 1st and 4th terms are obtained in the form
\[
T_{14}=\frac{(p+1)^2}{p^2+p+1}\left(p^{-1}-p^{-4k}\right).
\]
Similarly, the 2nd and 3rd terms are obtained in the form
\begin{align*}
 T_{23}=&-(p-1)(p+1)^2\frac{p^{3-4k}}{(p^2+p+1)}\frac{p^{4k-4}(p^{4k-4}-1)}{p^{4k-4}(p^4-1)}\\
 &\qquad +p^{-3}(p-1)^2(p+1)\frac{(p^{4k} - 1)(p^2+p+1)}{p^{4k-4}(p^4 - 1)(p^2+p+1)}\\
=&\frac{-(p-1)(p+1)^2p^{-1}(p^{4k-4}-1)+p^{-3}(p-1)^2(p+1)(p^{4k} - 1)(p^2+p+1)}{p^{4k-4}(p^4 - 1)(p^2+p+1)}\\
=&\frac{(p-1)(p+1)}{p^3}\frac{-(p+1)p^{2}(p^{4k-4}-1)+ (p-1)(p^{4k} - 1)(p^2+p+1)}{p^{4k-4}(p^4 - 1)(p^2+p+1)}.
\end{align*}
Since $p^4-1=(p-1)(p+1)(p^2+1)$, it is obtained that
\begin{align*}
 T_{23}
=&\frac{(p-1)(p+1)}{p^3}\frac{-(p+1)p^{2}(p^{4k-4}-1)+  (p^{4k} - 1)(p^3-1) }{p^{4k-4}(p^4 - 1)(p^2+p+1)}\\
=&\frac{p-1}{p}+\frac{1}{p^{4k-1}(p^2+p+1)}-\frac1{p^2+p+1}.
\end{align*}
Thus,
\begin{align*}
    S&=\frac{(p+1)^2}{p^2+p+1}\left(p^{-1}-p^{-4k}\right)+\frac{p-1}{p}+\frac{1}{p^{4k-1}(p^2+p+1)}-\frac1{p^2+p+1}+p^{-4k}\\
    &=\frac{1}{p}+\frac{p-1}{p}-\frac{1}{p^{4k}}+p^{-4k}=1.
\end{align*}
\end{proof}

```latex
\subsection{Uniform Distribution Lemma}

Let $m\in\mathbb Z_{p^k}$ and let $m=p^t u$, where $0<t<k$ and 
$u\in U(\mathbb Z_{p^k})$. 
The determinant fibers
\[
a_{p^k}(p^t u)
=
\left|
\{A\in M_2(\mathbb Z_{p^k}) : \det(A)=p^t u\}
\right|
\]
are considered.
It is observed that the map
\[
U(\mathbb Z_{p^k})\longrightarrow \{p^t u : u\in U(\mathbb Z_{p^k})\},
\qquad
u\mapsto p^t u,
\]
is not injective. For example, if $p=3$, $k=2$ and $t=1$, then
\[
U(\mathbb{Z}_9)=\{1,2,4,5,7,8\},
\]
that is, the elements which are not divisible by $3$.
 In this case,
\[
\begin{array}{c|c}
u & 3u\mod9\\
\hline
1 & 3\\
2 & 6\\
4 & 3\\
5 & 6\\
7 & 3\\
8 & 6
\end{array}
\]
Thus, the images of $u\mapsto 3u$ are only $\{3,6\}$, so these units produce only two different determinants. In general, it is observed that
\[
p^t u_1 = p^t u_2
\iff
u_1 \equiv u_2 \pmod{p^{k-t}},
\]
which shows that different determinant values correspond exactly to residue classes modulo $p^{k-t}$; hence there are
\[
|U(\mathbb Z_{p^{k-t}})|=p^{k-t-1}(p-1)
\]
distinct values.\\

In the previous example, $1\equiv 4\equiv 7\mod 3$ and $2\equiv 5\equiv 8 \mod 3$, and precisely these elements give the same determinant. Thus,
\[
U(\mathbb{Z}_3)=\{ \widetilde{1},\widetilde{2}\}.
\]
The class $\widetilde{1}$ can be lifted modulo $9$ in three ways: $1,4,7$, and the class $\widetilde{2}$ can be lifted as $2,5,8$. Each class has $p=3$ lifts. The choice of the lift does not matter – all representatives of the same residue class give the same determinant. The same property holds in general.

\begin{lemma}\label{nostolemma}
Each unit residue class modulo $p^{k-t}$ has exactly $p^t$ lifts 
in the unit group $U(\mathbb Z_{p^k})$.    
\end{lemma}

\begin{proof}
The natural reduction map
\[
\pi: U(\mathbb Z_{p^k}) \longrightarrow U(\mathbb Z_{p^{k-t}}),
\qquad
\pi(u)= u \bmod p^{k-t}
\]
is considered.
The map is well-defined and surjective. Its kernel is
\[
\ker(\pi) = \{ u \in U(\mathbb Z_{p^k}) : u \equiv 1 \mod{p^{k-t}} \}
= 1 + \langle p^{k-t}\rangle.
\]
By the first isomorphism theorem, the group isomorphism
\[
U(\mathbb Z_{p^k})/(1 + \langle p^{k-t}\rangle) \cong U(\mathbb Z_{p^{k-t}})
\]
is obtained.
In particular, each fiber $\pi^{-1}(a)=a\cdot \ker(\pi)$, where $a\in U(\mathbb Z_{p^{k-t}})$, is isomorphic to the kernel $\ker(\pi)$, and therefore its cardinality is
\[
|\pi^{-1}(a)| = |\ker(\pi)|.
\]
It is therefore sufficient to calculate the cardinality of the kernel. The kernel consists of those units $u\in U(\mathbb Z_{p^k})$ for which $u = 1 + p^{k-t}m$ for some $m\in \mathbb Z_{p^k}$. Such an element is a unit, because
\[
u \equiv 1 \pmod p,
\]
and hence $\gcd(u,p)=1$. Two different values $m_1$ and $m_2$ give the same element modulo $p^k$ precisely when
\[
p^{k-t}(m_1-m_2) \equiv 0 \pmod{p^k},
\]
which is equivalent to the condition $p^t \mid (m_1-m_2)$. Thus, $m$ is uniquely determined modulo $p^t$. Since $\mathbb Z_{p^t}$ contains exactly $p^t$ elements, it is obtained that
\[
|\ker(\pi)| = p^t.
\]
Consequently, each fiber $\pi^{-1}(a)$ contains exactly $p^t$ elements, which proves the claim.
\end{proof}
The following lemma is a central tool in the computation of determinant fibers. It shows that within each Smith class, the unit part of the determinant is uniformly distributed.

\begin{lemma}[Uniform Distribution in a Smith Class]
\label{lem:tasajakautuminen}
For each unit residue class \(\widetilde{u}\in U(\mathbb Z_{p^{k-t}})\), that is, for \(u\in U(\mathbb Z_{p^k})\) modulo \(p^{k-t}\), define
\[
\mathcal O_{\widetilde{u}}(D_{i,j})
=
\{
A\in\mathcal O(D_{i,j})
:
\det(A)=p^t u
\},
\]
where $u$ is an arbitrary lift of $\widetilde{u}$ to the ring \(\mathbb Z_{p^k}\).
This definition is well-defined, because \(p^t u\) depends only on the residue class of $u$ modulo \(p^{k-t}\).\\
\\
The sets \(\mathcal O_{\widetilde{u}}(D_{i,j})\) form a partition of the orbit
\(\mathcal O(D_{i,j})\):
\[
\mathcal O(D_{i,j})=\bigsqcup_{\widetilde{u}\in U(\mathbb Z_{p^{k-t}})}
\mathcal O_{\widetilde{u}}(D_{i,j}),
\]
and for all \(\widetilde{u},\widetilde{v}\in U(\mathbb Z_{p^{k-t}})\) it holds that
\[
|\mathcal O_{\widetilde{u}}(D_{i,j})|
=
|\mathcal O_{\widetilde{v}}(D_{i,j})|.
\]
In particular,
\[
|\mathcal O_{\widetilde{u}}(D_{i,j})|
=
\frac{|\mathcal O(D_{i,j})|}
{|U(\mathbb Z_{p^{k-t}})|},
\]
which is the same for all \(\widetilde{u}\in U(\mathbb Z_{p^{k-t}})\).
\end{lemma}

\begin{proof}
Let \(\widetilde{u},\widetilde{v}\in U(\mathbb Z_{p^{k-t}})\). Choose lifts \(u,v\in U(\mathbb Z_{p^k})\). Since the determinant map
\(\det:\mathrm{GL}_2(\mathbb Z_{p^k})\to U(\mathbb Z_{p^k})\) is surjective, a matrix
\(T\in\mathrm{GL}_2(\mathbb Z_{p^k})\) can be chosen such that \(\det(T)=v u^{-1}\).
Define the map
\[
\Phi_T:\mathcal O(D_{i,j})\to \mathcal O(D_{i,j}),\qquad
A\mapsto TA.
\]
Since left multiplication by an invertible matrix is a bijection and
\(TA=(TU)D_{i,j}V^{-1}\), the map \(\Phi_T\) is a bijection of the orbit. Moreover, if \(A\in \mathcal O_{\bar{u}}(D_{i,j})\), that is, \(\det(A)=p^t u\), then
\[
\det(\Phi_T(A))=\det(TA)=\det(T)\det(A)=v u^{-1}\, p^t u = p^t v.
\]
Thus \(A\in\mathcal O_{\widetilde{u}}(D_{i,j})\) if and only if
\(\Phi_T(A)\in\mathcal O_{\widetilde{v}}(D_{i,j})\). Since the above construction can be performed for every pair of lifts $(u,v)$, the map \(\Phi_T\) induces a bijection
\(\mathcal O_{\widetilde{u}}(D_{i,j})\to\mathcal O_{\widetilde{v}}(D_{i,j})\). Consequently,
\[
|\mathcal O_{\widetilde{u}}(D_{i,j})|=|\mathcal O_{\widetilde{v}}(D_{i,j})|
\]
for all \(\widetilde{u},\widetilde{v}\in U(\mathbb Z_{p^{k-t}})\).\\
\\
For every \(A=UD_{i,j}V^{-1}\in\mathcal O(D_{i,j})\), one has
\[
\det(A)=p^{i+j}\det(U)\det(V)^{-1}=p^t\cdot \frac{\det(U)}{\det(V)}.
\]
Since \(\det(U),\det(V)\in U(\mathbb Z_{p^k})\), it follows that
\(\det(A)=p^t u\) for some
\(u\in U(\mathbb Z_{p^k})\). However, the map \(u\mapsto p^t u\) is not injective; its value depends only on the residue class of \(u\) modulo \(p^{k-t}\). Thus, each determinant value corresponds exactly to one residue class
\(\widetilde{u}\in U(\mathbb Z_{p^{k-t}})\). Therefore,
\[
\mathcal O(D_{i,j})=\bigsqcup_{\widetilde{u}\in U(\mathbb Z_{p^{k-t}})}
\mathcal O_{\widetilde{u}}(D_{i,j})
\]
is a disjoint union.
Since all sets \(\mathcal O_{\widetilde{u}}(D_{i,j})\) have the same cardinality and there are
\(|U(\mathbb Z_{p^{k-t}})|\) such sets, one obtains
\[
|\mathcal O(D_{i,j})|
=
|U(\mathbb Z_{p^{k-t}})|\,
|\mathcal O_{\widetilde{u}}(D_{i,j})|.
\]
This yields the desired formula for the sizes of the fibers.
\end{proof}

\subsection{Fiber Sum Formulas}

Let \(t=v_{p}(m)\) and \(m=p^tu\). Then, according to the previous section, there are exactly
\[
 |U(\mathbb Z_{p^{k-t}})|=p^{k-t-1}(p-1)
\]
distinct determinant values. In other words, the value \(a_{p^k}(p^t u)\) depends only on the residue class of
\(u\) modulo \(p^{k-t}\); 
correspondingly, the indexing can be performed directly by
\(\widetilde{u}\in U(\mathbb{Z}_{p^{k-t}})\), where \(u\in U(\mathbb Z_{p^k})\) 
is an arbitrary lift of this class to the ring \(\mathbb Z_{p^k}\).\\
\\
Lemma \ref{lem:tasajakautuminen} guarantees that in each Smith class
\(\mathcal O(D_{i,j})\) there are exactly
\[
|\mathcal O_{\widetilde{u}}(D_{i,j})|
=
\frac{|\mathcal O(D_{i,j})|}{|U(\mathbb Z_{p^{k-t}})|}
\]
matrices whose determinant is \(p^t u\), where \(u\) is any lift of the class \(\widetilde{u}\). 
Since the different Smith classes are disjoint, one obtains
\[
a_{p^k}(p^t u)
=\sum_{\substack{0\le i\le j<k\\ i+j=t}}
|\mathcal O_{\widetilde{u}}(D_{i,j})|
=\sum_{\substack{0\le i\le j<k\\ i+j=t}}
\frac{|\mathcal O(D_{i,j})|}{|U(\mathbb Z_{p^{k-t}})|}.
\]
The invariance of the fiber can be seen from the previous expression. The element \(u\) can be any lift of \(\widetilde{u}\), but the latter sums remain invariant. Next, define
\[
B_t = \sum_{\substack{0\le i\le j<k\\ i+j=t}} |\mathcal O(D_{i,j})|,
\qquad 0\le t<k.
\]
Then the final sum formula is obtained as
\begin{align}\label{KuitSYmKaav}
a_{p^k}(p^t u)
=
\frac{B_t}{|U(\mathbb Z_{p^{k-t}})|}
=
\frac{B_t}{p^{k-t-1}(p-1)}.
\end{align}
The complement formula for the zero fiber, Lemma \ref{NollakuidunKkaava}, gives
\[
a_{p^k}(0)
=
p^{4k}
-
\sum_{t=0}^{k-1}
p^{k-t-1}(p-1)\, A_t,
\]
where $A_t = a_{p^k}(m)$, and $v_p(m)=t$. Thus, in the case \(t=k\), i.e. \(m=0\), formula (\ref{KuitSYmKaav}) gives
\[
a_{p^k}(0)
=
p^{4k}
-
\sum_{t=0}^{k-1}B_t.
\]
In Proposition \ref{Uksikkokuidutmaara}, it has been proved that
\[
a_{p^k}(m)=p^{3k-2}(p^2-1).
\]
in the case \(t=0\), i.e. when \(m\in U({\mathbb{Z}_{p^k}})\). In summary, a sum representation for the sizes of the fibers is obtained in terms of the sizes of the orbits:

\begin{lemma} \label{alustavalemma}
Let \(p\) be an odd prime, \(k\ge 1\), and \(m\in \mathbb Z_{p^k}\).
Denote \(t=v_p(m)\) (where \(v_p(0)=k\)). Then the fiber size
\[
a_{p^k}(m)=\left|\{A\in M_2(\mathbb Z_{p^k}):\det(A)=m\}\right|
\]
is
\[
a_{p^k}(m)=
\begin{cases}
p^{3k-2}(p^2-1), & t=0,\\[6pt]
\dfrac{B_t}{p^{k-t-1}(p-1)}, & 0<t<k,\\[8pt]
p^{4k}-\displaystyle\sum_{t=0}^{k-1} B_t, & t=k,
\end{cases}
\]
where \(B_t\) (for \(0\le t<k\)) is defined by
\[
B_t = \sum_{\substack{0\le i\le j\\ i+j=t}} |\mathcal O(D_{i,j})|,
\]
and the orbit sizes are
\[
|\mathcal O(D_{i,j})|=
\begin{cases}
p^{4k-3-4i}(p-1)^2(p+1), & i=j<k,\\[4pt]
p^{4k-4-3i-j}(p-1)^2(p+1)^2, & i<j<k,\\[4pt]
p^{3k-3i-3}(p-1)(p+1)^2, & i<k,\ j=k.
\end{cases}
\]
\end{lemma}
Let \(0<t<k\) and define
\[
\delta(t)=\begin{cases}
1, & \text{when }\ 2|t,\\
0, & \text{otherwise}.
\end{cases}
\]
Then
\begin{align}
    a_{p^k}(m)&=\dfrac{B_t}{p^{k-t-1}(p-1)}\nonumber\\
    &=\dfrac{1}{p^{k-t-1}(p-1)}\sum_{\substack{0\le i\le j\\ i+j=t}} |\mathcal O(D_{i,j})|\nonumber\\
    &=\dfrac{1}{p^{k-t-1}(p-1)}\Big(\sum_{\substack{0\le i< j\\ i+j=t}} |\mathcal O(D_{i,j})|+ \delta(t)|\mathcal O(D_{t/2,t/2})|\Big)\label{SYYmmaa}
\end{align}
The following lemma is proved next:

\begin{lemma}
Let \(p>1\) and let \(t\ge 0\) be an integer. Then

\[
\sum_{\substack{0\le i< j\\ i+j=t}} p^{-3i-j}
=
p^{-t}\,
\frac{1-p^{-2\lfloor (t+1)/2 \rfloor}}{1-p^{-2}}.
\]
\end{lemma}
\begin{proof}

Since \(i<j\) and \(i+j=t\), it follows that \(i < t/2\). Therefore,
\[
i=0,1,\dots, \left\lfloor \frac{t-1}{2} \right\rfloor.
\]
Consequently,
\[
\sum_{\substack{0\le i< j\\ i+j=t}} p^{-3i-j}
=
\sum_{i=0}^{\lfloor (t-1)/2 \rfloor} p^{-3i-(t-i)}
= p^{-t} \sum_{i=0}^{\lfloor (t-1)/2 \rfloor} (p^{-2})^i.
\]
This is a geometric series, and since
\[
\left\lfloor \frac{t-1}{2} \right\rfloor + 1
=
\left\lfloor \frac{t+1}{2} \right\rfloor,
\]
it follows that
\[
\sum_{i=0}^{\lfloor (t-1)/2 \rfloor} (p^{-2})^i
=
\frac{1-p^{-2\lfloor (t+1)/2 \rfloor}}{1-p^{-2}}.
\]
\end{proof}
Applying the previous lemma and the orbit size formulas, formula (\ref{SYYmmaa}) can be simplified. Let \(t\) be odd, in which case
\begin{align*}
    a_{p^k}(m) 
    &=\dfrac{(p-1)^2(p+1)^2p^{4k-4}}{p^{k-t-1}(p-1)} \sum_{\substack{0\le i< j\\ i+j=t}} p^{-3i-j}\\
    &=(p-1)(p+1)^2p^{3k+t-3} p^{-t}\,
\frac{1-p^{-2\lfloor (t+1)/2 \rfloor}}{1-p^{-2}}
\end{align*}
When \(t\) is odd, one has \(\lfloor (t+1)/2 \rfloor=(t+1)/2\). Thus,
\begin{align*}
    a_{p^k}(m) 
    &=(p-1)(p+1)^2p^{3k-3} \frac{1-p^{- (t+1) }}{1-p^{-2}}.
\end{align*}
Since \(p^2-1=(p+1)(p-1)\), one obtains \(1-{p^{-2}}=(p+1)(p-1)p^{-2}\), and hence
\begin{align*}
    a_{p^k}(m) 
    &=(p-1)(p+1)^2p^{3k-3} p^2\frac{1-p^{- (t+1) }}{(p+1)(p-1)}\\
    &=(p+1)p^{3k-1} (1-p^{-t-1 })
\end{align*}
Hence
 \begin{align}
    a_{p^k}(m) 
&=(p+1)p^{3k-1} p^{-t-1}(p^{t+1}-1)\nonumber\\
    &=(p+1)p^{3k-t-2}(p^{t+1} -1). \label{paritontapasu}
\end{align}
Assume next that \(t\) is even. Then
\begin{align*}
    a_{p^k}(m) 
    &=\dfrac{1}{p^{k-t-1}(p-1)}\Big(\sum_{\substack{0\le i< j\\ i+j=t}} |\mathcal O(D_{i,j})|+  |\mathcal O(D_{t/2,t/2})|\Big)\\
    &=\dfrac{1}{p^{k-t-1}(p-1)}\Big((p-1)^2(p+1)^2p^{4k-4}\sum_{\substack{0\le i< j\\ i+j=t}} p^{-3i-j}+  p^{4k-3-2t}(p-1)^2(p+1)\Big).
\end{align*}
Applying the previous summation lemma and the identity \(1-{p^{-2}}=(p+1)(p-1)p^{-2}\), one obtains
\begin{align*}
    a_{p^k}(m) 
    &=\dfrac{1}{p^{k-t-1}(p-1)}\Big((p-1)^2(p+1)^2p^{4k-4}p^{-t}\,
\frac{1-p^{-2\lfloor (t+1)/2 \rfloor}}{1-p^{-2}}+  p^{4k-3-2t}(p-1)^2(p+1)\Big)\\
&=\dfrac{1}{p^{k-t-1}(p-1)}\Big((p-1)^2(p+1)^2p^{4k-4}p^{2-t}\,
\frac{1-p^{-2\lfloor (t+1)/2 \rfloor}}{(p+1)(p-1) }+  p^{4k-3-2t}(p-1)^2(p+1)\Big)\\
&=\dfrac{1}{p^{k-t-1}(p-1)}\Big((p-1)(p+1)p^{4k-4}p^{2-t}\,
(1-p^{-2\lfloor (t+1)/2 \rfloor})+  p^{4k-3-2t}(p-1)^2(p+1)\Big).
\end{align*}
When \(t\) is even, one has \(\lfloor (t+1)/2 \rfloor=\frac{t}{2}\), and therefore
\begin{align*}
    a_{p^k}(m) 
&=\dfrac{1}{p^{k-t-1}(p-1)}\Big((p-1)(p+1)p^{4k-t-2}
(1-p^{-t})+  p^{4k-3-2t}(p-1)^2(p+1)\Big)\\
&=\dfrac{1}{p^{k-t-1}(p-1)}\Big((p-1)(p+1)p^{4k-t-2}
-(p-1)(p+1)p^{4k-2t-2}
+  p^{4k-3-2t}(p-1)^2(p+1)\Big)\\
&=\dfrac{1}{p^{k-t-1}(p-1)}\; (p-1)(p+1)p^{4k-3-2t}\big(p^{t+1}
- p
+  p-1  \big)\\
&= (p+1)p^{3k-t-2}(p^{t+1}-1),
\end{align*}
which is the same formula as in the odd case (\ref{paritontapasu}). Therefore, when \(v_p(m)=t\) and \(0<t<k\), one has
\begin{align}\label{Kaavai<j}
 a_{p^k}(m)= p^{3k-2-t}(p+1)(p^{t+1}-1).
\end{align}
Let us finally consider the case \(t=k\), when \(\det(A)=m=0\). Then 
\[
a_{p^k}(0)= p^{4k}-\displaystyle\sum_{t=0}^{k-1} B_t.
\]
From the case \(0\le t<k\), the summation formula (\ref{KuitSYmKaav}) gives
\[
a_p(m)=\dfrac{B_t}{p^{k-t-1}(p-1)}\iff B_t=p^{k-t-1}(p-1)a_p(m)
\]
and when \(t=0\), then
\[
a_{p^k}(m)=p^{3k-2}(p^2-1).
\]
Since the summation formula (\ref{KuitSYmKaav}) is in particular valid also for \(t=0\), it follows that
\begin{align*}
 B_0&=a_p(m)p^{k-1}(p-1)\\
&=p^{3k-2}(p^2-1)p^{k-1}(p-1)\\
&=p^{4k-3}(p^2-1)(p-1)
\end{align*}
For the correctness of the computations, it is useful to note that this can also be calculated as
\[
B_0=|\mathcal O(D_{0,0})| =|\mathrm{GL}_2(\mathbb{Z}_{p^k})|= p^{4k-3}(p-1)^2(p+1)=p^{4k-3}(p^2-1)(p-1),
\]
since all invertible matrices belong to the same orbit. When $0<i<k$, it was calculated by (\ref{Kaavai<j}) that
\begin{align*}
    B_t&=p^{k-t-1}(p-1)a_p(m)\\
    &= p^{k-t-1}(p-1)p^{3k-2-t}(p+1)(p^{t+1}-1)\\
    &=p^{4k-3-2t}(p^2-1)(p^{t+1}-1)\\
    &=p^{4k-2-t}(p^2-1)-p^{4k-3-2t}(p^2-1)
\end{align*}
Thus,
\begin{align*}
  a_{p^k}(0)&= p^{4k}-B_0- \sum_{t=1}^{k-1} B_t\\
  &=p^{4k-3}(p^2-1)(p-1)- (p^2-1)p^{4k-2}\sum_{t=1}^{k-1} p^{-t}+(p^2-1)p^{4k-3}\sum_{t=1}^{k-1}p^{-2t}.
\end{align*}
When \(\delta \ge 1\), the geometric sum is obtained as
\begin{align}\label{geomsar}
\sum_{t=1}^{k-1} p^{-\delta t}
=
\frac{p^{-\delta} - p^{-\delta k}}{1 - p^{-\delta}}
=
\frac{1 - p^{-\delta(k-1)}}{p^{\delta} - 1},
\end{align}
by means of which it is obtained that
\begin{align}
  a_{p^k}(0)&= p^{4k}-B_0- \sum_{t=1}^{k-1} B_t\nonumber\\
  &=p^{4k}-p^{4k-3}(p^2-1)(p-1)- (p^2-1)p^{4k-2}\frac{1 - p^{-(k-1)}}{p - 1}+(p^2-1)p^{4k-3}\frac{1 - p^{-2(k-1)}}{p^{2} - 1}\nonumber\\
  &=p^{4k}-p^{4k-3}(p^3-p^2-p+1)- (p+1)p^{3k-1}(p^{k-1} - 1) +p^{2k-1 }(p^{2k-2} - 1)\nonumber\\
  &=p^{4k}-p^{4k}+p^{4k-1}+p^{4k-2}-p^{4k-3}- (p+1)(p^{4k-2}  - p^{3k-1}) +p^{4k-3} - p^{2k-1 }\nonumber\\
  &=p^{4k-1} -p^{4k-3}+ p^{3k}-p^{4k-1}+p^{3k-1}  +p^{4k-3} - p^{2k-1 }\nonumber\\
  &= p^{3k}+p^{3k-1}  - p^{2k-1 }\nonumber\\
  &=p^{2k-1}(p^{k+1}+p^k-1).\label{Tapausk=Nolla}
\end{align}
Thus, the main theorem is proved.

\end{document}